\documentclass[12pt]{amsart}
\usepackage{mathrsfs}
\usepackage{amsmath}
\usepackage{amssymb}
\usepackage{amsfonts}
\usepackage{amsopn}
\usepackage{amsthm}
\usepackage{latexsym}
\usepackage[all]{xy}
\usepackage{enumerate}
\usepackage{geometry}
\usepackage{fancyhdr}
\usepackage{graphicx}
\usepackage{wrapfig}
\usepackage{float}

\usepackage[
    backend=biber,
style=alphabetic,
sorting=nyt,
]{biblatex}
\usepackage[]{hyperref}
\hypersetup{
    colorlinks=true,
}

\newtheorem{thm}{Theorem}

\newtheorem*{mkthm}{Monge-Kantorovich Duality}
\newtheorem{lem}{Lemma}
\newtheorem{prop}{Proposition}

\theoremstyle{definition}
\newtheorem{dfn}{Definition}
\newtheorem{exx}{Example}
\theoremstyle{remark}

\newcommand{\bR}{\mathbb{R}}

\newcommand{\bZ}{\mathbb{Z}}

\newcommand{\del}{\partial}

\newcommand{\ysub}{\del^c \psi(y)}
\newcommand{\sub}{\del^c \psi^c(x')}

\newcommand{\hh}{\hookleftarrow}

\newcommand{\sZ}{\mathscr{Z}}
\newcommand{\cd}{c_\Delta}
\newcommand{\sH}{\mathscr{H}}
\begin{document}

\title{Topology of Singularities of Optimal Semicouplings}

\author{J. H. Martel}
\date{\today}
\email{jhmartel@protonmail.com}
\maketitle

\begin{abstract}
This article is a summary of some results from the author's thesis \cite{martel}. We study the topology of singularities of $c$-optimal semicouplings in unequal dimension. Our main results \ref{A}, \ref{B} describe homotopy-reductions from a source space $(X,\sigma)$ onto the singularities $Z_j$, $j\geq 0$ of $c$-optimal semicouplings whenever $(Y, \tau)$ is a Riemannian target space and $c: X\times Y\to \bR \cup \{+\infty\}$ is a cost satisfying some general assumptions (A0)--(A5). We construct continuous strong deformation retracts $X\leadsto Z_j$ whenever a condition called Uniform Halfspace (UHS) condition is satisfied along appropriate subsets.

\end{abstract}
\tableofcontents

\section{Introduction}

%[Notation] 

Throughout this article $X, Y$ designate finite-dimensional Riemannian manifolds-with-corners. The spaces $X, Y$ are typically of unequal dimension $\dim(X) >> \dim(Y)$. We let $\sigma, \tau$ denote Borel-Radon measures on $X,Y$, called the \emph{source} and \emph{target} measures, respectively. We assume that $\sigma, \tau$ are absolutely continuous with respect to the Riemannian volume measures on $X,Y$, or absolutely continuous with respect to the Hausdorff measures $\mathscr{H}_X$, $\mathscr{H}_Y$. 
%
% The present thesis develops a bridge between measure theory and algebraic topology, as inspired from Kantorovich duality in optimal transport. The bridge is categorical, being a contravariant functor $Z=Z(c, \sigma, \tau)$ defined by a cost $c:X\times Y \to \bR$ between source $(X,\sigma)$ and target $(Y,\tau)$ measure spaces. If $2^X, 2^Y$ denote the category of closed subsets of $X,Y$ respectively, then $Z$ can be represented as a contravariant correspondance $Z: 2^Y \to 2^X$ between closed topological subsets of $Y$ and $X$, when $c$ satisfies some standard geometric assumptions. 

%%ABSTRACT:?
Optimal transportation is a subject with a remarkably wide range of applications \cite{Vil1}, \cite{Vil2}. In \cite{martel} some new applications of optimal transport to algebraic topology are proposed. The idea is that the singularities $Z$ of $c$-optimal semicouplings define contravariant topological functors $Z: 2^Y \to 2^X$, where $Z(Y_I):=\cap_{y\in Y_I} \ysub$ for $\psi^{cc}=\psi$ a $c$-concave potential on $Y$. The functor $Z=Z(c,\sigma, \tau)=Z(\psi)$ depends on the optimal transport datum $c, \sigma, \tau$, and specifically solutions $\psi: Y \to \bR \cup \{-\infty\}$ to Kantorovich's dual max program. We study the topology of the functor $Z$, which exists in general conditions whenever the source $\sigma$ is \emph{abundant} with respect to the target $\tau$:  \begin{equation}\label{abundant}\int_X \sigma > \int_Y \tau. \end{equation} 

The results of this article demonstrate that homotopy-isomorphisms and strong deformation retracts can be constructed between various inclusions $Z(Y_I) \hookleftarrow Z(Y_J)$ arising from contravariance and inclusions $Y_I\hookrightarrow Y_J$. Following an idea of Dror Bar-Natan \cite{Bar2002} the functor $Z$ leads to a descending filtration of $X$ into subvarieties $Z_j$, $j=0,1,2, \ldots$, namely \begin{equation}\label{filt}
(X=Z_0) \hookleftarrow (A=Z_1) \hookleftarrow Z_2 \hookleftarrow Z_3 \hookleftarrow \cdots .
\end{equation}
Our first theorem \ref{A} identifies local conditions for which the first inclusion $(X=Z_0) \hookleftarrow (A=Z_1)$ is a homotopy-isomorphism, even a strong deformation retract. Our second theorem \ref{B} identifies local conditions and maximal index $J\geq 0$ for which the right-hand inclusions $Z_j \hookleftarrow Z_{j+1}$ are homotopy-isomorphisms, and again even strong deformation retracts. The results of \ref{A}, \ref{B}, \ref{retract1}, \ref{retract2} identifies a condition, which we call Uniform Halfspace (UHS) conditions \ref{whs}, and indices $J\geq 0$ for which the source space $X$ can be continously reduced via strong deformation retracts to codimension-$J$ subvarieties $X \leadsto Z_{J+1} $. This leads to the possibility of constructing spines and souls as the singularities of $c$-optimal semicouplings, a topological application which we cannot elaborate here.

\section{Cost Assumptions}\label{costassump}

%Topology forces the nonexistence of continuous retracts. 
 The topology of the locus-of-discontinuities and the singularity $Z$ of $c$-optimal semicouplings is determined by the geometry of the cost $c$, and the geometry of $c$-concave potentials. The cost $c=c(x,y)$ represents the cost of transporting a unit mass at source $x$ to target at $y$. The proofs of our Theorems \ref{A}, \ref{B}, \ref{retract1}, \ref{retract2} require cost functions $c: X\times Y \to \bR \cup \{+\infty\}$ satisfying important assumptions labelled (A0), $\ldots$, (A6). Abbreviate $c_y(x):=c(x,y)$. The assumptions are the following:

\begin{itemize}
\item[\textbf{(A0)}] The cost $c:X\times Y \to \bR \cup \{+\infty\}$ is continuous throughout $dom(c) \subset X\times Y$ and uniformly bounded from below, e.g. $c\geq 0$. Moreover we assume the sublevels $\{x\in X~|~ c(x,y) \leq t\}$ are compact subsets of $X$ for every $t\in \bR$, $y\in Y$. Thus we assume $x\mapsto c(x,y)$ is \emph{coercive} for every $y\in Y$.

\item[\textbf{(A1)}] The cost is twice continuously differentiable with respect to the source variable $x$, uniformly in $y$ throughout $dom(c)$. So for every $y\in Y$, the Hessian function $x\mapsto \nabla_{xx}^2 c(x,y)$ exists and is continuous throughout $dom(c_y)$. 

\item[\textbf{(A2)}] The function $(x,y)\mapsto ||\nabla_x c(x,y)||$ is upper semicontinuous throughout $dom(c)$. So for every $t\in \bR$ the superlevel set $\{ ||\nabla_x c(x,y)|| \geq t\}$ is a closed subset of $dom(c)$. 

\item[\textbf{(A3)}] For every $y\in Y$, we assume $x'\mapsto \nabla_x c(x',y)$ does not vanish identically on any open subset of $dom(c_y)$.  

\item[\textbf{(A4)}] The cost satisfies (Twist) condition with respect to the source variable throughout $dom(c)$. So for every $x'$ the rule $y\mapsto \nabla_x c(x',y)$ defines an injective mapping $dom(c_{x'}) \to T_{x'} X$. 

\item[\textbf{(A5)}]\label{A+} For every $x\in X$, the function $y\mapsto c(x,y)$ is continuously differentiable; and for every $y\in Y$, the gradients $\nabla_y c(x,y)$ are bounded on compact subsets $K\subset X$.
\end{itemize}

Consequences of Assumptions (A0), (A1), $\ldots$ will be elaborated below. In the simplest case where the source and target spaces $X,Y$ are compact, and $c$ is smooth and finite-valued throughout $X\times Y$, then Assumptions (A0)--(A2) are readily confirmed. Assumption (A3) forbids the cost $c(x,y)$ from being locally constant on any open subset of $X$. 

The Assumption (A4) implies the gradients of the cross-differences $\nabla_x \cd(x,y,y')$ are nonzero for distinct $y,y'$, where $\cd(x;y, y'):=c(x,y)-c(x,y')$ is the two-pointed cross difference. In practice, costs $c$ which have poles, e.g. $c(x,y)=+\infty$ when $x=y$, will more readily satisfy (A4). Indeed the functions $x\mapsto \cd(x;y, y')$ have critical points if $X$ is closed compact space and $\cd$ is everywhere finite. 

We will prove (A0)--(A4) implies the general uniqueness of $c$-optimal semicouplings. Peculiar to the semicoupling setting is (A4), which requires injectivity of mappings $y\mapsto \nabla_x c(x',y):Y\to T_{x'} X$ for every $x\in X$, where $\dim(Y)< \dim(X)$. The Assumption (A5) is useful in the construction of our deformation retracts in Theorems \ref{retract1}--\ref{retract2} below. 

A further assumption (A6) is necessary for the deformation retracts constructed in the present article. These retracts depend on the nonvanishing of averaged vector fields denoted $\eta_{avg}(x) \in T_x X$. The field $\eta_{avg}$ is an average of a $Y$-parameter family of potentials $\eta(x,y)$ defined with respect to $c$-concave potentials $\psi^{cc}=\psi$, and parameters $\beta\geq 2$, by formulas\begin{equation}\label{introetaavg}
\eta(x,y):=|\psi(y_0)-\psi(y)-\cd(x,y_0,y)|^{-\beta} \nabla_x (c(x,y_0)-c(x,y)) 
\end{equation} for $y_0\in \del^c \phi(x)$, $y\in Y$. %Compare \eqref{avgI}, \eqref{avgII}, \ref{whs}. 
The formal definition of $\eta_{avg}(x)$ depends on the setting. Typically there is a Radon measure $\bar{\nu}_x$ on $Y$, depending on $x$, absolutely continuous with respect to $\mathscr{H}_Y$, and with average \begin{equation}\label{introavg}
\eta_{avg}(x):=(\bar{\nu}_x[Y])^{-1} \int_Y \eta(x,y)d\bar{\nu}_x(y).
\end{equation}
With these definitions of $\eta_{avg}$ we can now state the assumption (A6): 
\begin{itemize}
\item[\textbf{(A6)}] 
The averaged vectors $\eta_{avg}(x)$ are bounded away from zero, uniformly with respect to $x\in Z'$ on the relevant subsets $Z'$ of $X$.
\end{itemize}
See \eqref{avgI}, \eqref{avgII}, \ref{whs} and the hypotheses of \ref{retract1}, \ref{retract2}. 

N.B. The assumption (A6) is defined relative to substs $Z'$ of $X$. In practice, the subsets $Z'$ will be  subsets of the form $Z'(x):=Z(\del^c \psi^c(x))$, $X-A$, $Z_j - Z_{j+1}$. More precise formulations of (A6) are given in \ref{a5spec}, \eqref{maxineq}, and the Uniform Halfspace (UHS) Conditions defined in \ref{whs}. The assumption (A6) depends on properties of $c$-concave potentials defined on $Y$, and is not an absolute assumption on the geometry of the cost $c$ like the previous (A0)--(A5).

\section{Optimal Semicouplings and Monge-Kantorovich Duality}\label{22-1}
Now we briefly introduce the optimal semicoupling program and Kantorovich's dual max program.  
\begin{dfn}[Semicoupling]
A \emph{semicoupling} between source $(X,\sigma)$ and target $(Y, \tau)$ is a Borel-Radon measure $\pi$ on the product space $X\times Y$ with $proj_Y \# \pi=\tau$ and $proj_X \# \pi \leq \sigma$. 
\end{dfn} The inequality $proj_X \# \pi \leq \sigma$ holds if for every Borel subset $O$, the numerical inequality $(proj_X \# \pi )[O]\leq \sigma[O]$ is satisfied. We remark that $\pi$ is a coupling between $\sigma$ and  $\tau$ when $proj_X \# \pi = \sigma$.

\begin{dfn}[$c$-transforms]
If $\psi: Y\to \bR \cup \{-\infty\}$ is any function on the target $Y$, then the $c$-Legendre transform $\psi^c: X\to \bR \cup\{+\infty\}$ is defined by $$\psi^c (x):=\sup_{y\in Y} [\psi(y)-c(x,y)], $$ for $x\in X$. 

If $\phi: X \to \bR \cup \{+\infty \} $ is any function on the source $X$, then the $c$-Legendre transform $\phi^c: Y\to \bR \cup \{-\infty\}$ is defined by the rule $$\phi^c(y)=\inf_{x\in X} [c(x,y) + \phi(x)],$$ for $y\in Y$. 
\end{dfn}

\begin{dfn}
[$c$-concavity]\label{c-concavity}
A function $\psi: Y\to \bR \cup \{-\infty\}$ is $c$-concave if $\psi^{cc}=(\psi^{c})^c$ coincides pointwise with $\psi$. Equivalently $\psi$ is $c$-concave if there exists a lower semicontinuous function $\phi: X\to \bR \cup \{+\infty\}$ such that $\phi^c = \psi$ pointwise. 
\end{dfn}

The above definitions imply $\psi, \psi^c$ satisfy the pointwise inequality 
\begin{equation} 
-\psi^c (x)+\psi(y) \leq c(x,y) \label{cal}
\end{equation}
 for all $x \in X$, $y\in Y$. The inequality \eqref{cal} and especially the case of equality is very important for our applications.
\begin{dfn}
[$c$-subdifferential]\label{subdif}
Let $\psi: Y\to \bR \cup \{-\infty\}$ be $c$-concave potential $\psi^{cc}=\psi$. Select $y_0\in Y$ where $\psi(y_0)$ is finite-valued. The subdifferential $\del^c \psi(y_0) \subset X$ consists of those points $x'\in X$ such that $$-\psi^c(x')+\psi(y_0) = c(x', y_0).$$ Or equivalently such that for all $y\in Y$, $$\psi(y)-c(x',y) \leq \psi(y_0)-c(x',y_0).$$
\end{dfn}

The main technical result we need is Monge-Kantorovich duality. Let $SC(\sigma, \tau) \subset \mathscr{M}_{\geq 0} (X \times Y)$ denote the set of all semicoupling measures $\pi$ between source $\sigma$ and target $\tau$.
\cite[Theorem 5.10, pp.57]{Vil1}.
\begin{mkthm}\label{maxI}
Let $c: X\times Y \to \bR\cup\{+\infty\}$ be a cost satisfying Assumptions (A0)--(A5). Let $\sigma, \tau$ be source and target measures on $X$, $Y$, respectively. Assume $\sigma, \tau$ are absolutely continuous with respect to Hausdorff measures $\sH_X$, $\sH_Y$, respectively. 

Then there exists unique closed domain $A\hookrightarrow X$ called the activated domain such that the unique $c$-optimal semicoupling $\pi_{opt}$ defines a coupling from $1_A.\sigma$ to $\tau$. 

Moreover there exists $c$-concave potentials $\psi^{cc}=\psi$ on $Y$ and $c$-convex potentials $\phi=\psi^c$ on $X$ solving Kantorovich's dual max program: 

\begin{equation}
\max_{\psi ~c\text{-concave}} [\int_A -\psi^c(x) d\sigma(x) + \int_Y \psi(y) d\tau(y)]= \min_{\pi\in SC(\sigma, \tau)} [ \int_{X\times Y} c(x,y) d\pi(x,y)] 
\label{mkdualeq}\end{equation}

Moreover the $c$-optimal semicoupling $\pi$ is supported on the graph of the $c$-subdifferentials $\del^c \psi^c(x)$, hence equality $-\psi^c(x)+\psi(y)=c(x,y)$ holds $\pi$-a.e.

\end{mkthm}

Assumptions (A0), $\ldots$, (A4), on the cost $c$ imply various properties of $c$-convex potentials and $c$-subdifferentials. The first useful property is that $c$-subdifferentials are nonempty wherever the potentials $\phi(x)$ or $\psi(y)$ are finite, see Lemma \ref{psilem1} below. Recall the domain of $\phi$ is defined $dom(\phi):=\{x\in X$ $| \phi(x) <+\infty\}$.

\begin{lem}\label{psilem1}
Let $c: X\times Y \to \bR$ be a cost satisfying Assumptions (A0)--(A2). Let $\psi: Y\to \bR \cup \{-\infty\}$ be a $c$-concave potential $(\psi^c)^c=\psi$. Abbreviate $\phi=\psi^c$. Suppose there exists $y' \in Y$ such that $\psi(y) \neq -\infty$. Then:

(i) $\psi$ is an upper semicontinuous function; and

(ii) $\del^c \psi(y)$ is a nonempty closed subset of $X$ for every $y\in dom(\psi)$; and 

(iii) $\phi$ is lower semicontinuous function; and 

(iv) $\del^c \psi(y)$ is a nonempty closed subset of $Y$ for every $x\in dom(\phi)$. 

\end{lem}
\begin{proof}
We omit the standard arguments.

\end{proof}

%\begin{figure}
%\centering
%\includegraphics[width=0.6\textwidth]{sep-active1.jpg}
%\caption{Disconnected Active Domain}
%\end{figure}

%\begin{figure}
%\centering
%\includegraphics[width=0.6\textwidth]{sep-active.jpg}
%\caption{Connected Active Domain when $mass[\sigma]/mass[\tau]\approx 1^+ $}
%\end{figure}

\section{(Twist) Condition}\label{uos}
%% TWIST begins below.
Thus far we have established the existence of optimal semicouplings and uniqueness of active domains. Now we describe the (Twist) hypothesis and the uniqueness of optimal couplings when the source measure $\sigma$ is absolutely continuous with respect to the reference source measure $\mathscr{H}_X^d$ in $X$. 

The next definition elaborates Assumption (A4) from \S \ref{costassump}.
\begin{dfn}[\textbf{(Twist)}]\label{twistdef}
Let $c: X\times Y \to \bR$ be cost function satisfying Assumptions (A0)--(A1). Then $c$ satisfies (Twist) condition if for every $x'\in X$ the rule $$y\mapsto \nabla_x c(x',y) $$ defines an injective mapping $\nabla_x c(x',\cdot):dom(c_{x'})\to T_{x'} X$. 
\end{dfn}

Observe that (Twist) condition is equivalent to the function $$x\mapsto \cd(x;y_0, y_1):=c(x,y_0)-c(x,y_1)$$ admitting no critical points on $X$, whenever $y_0, y_1 \in Y$ are distinct. If $X$ is closed manifold, then $x\mapsto \cd(x;y_0,y_1)$ certainly has critical points which violates (Twist). Our settings assume $X$ is a manifold-with-corners with nontrivial boundary $\del X \neq \emptyset$. The (Twist) condition requires $\cd$ admit no critical points on the interior of $X$, and all maxima/minima exist on the boundary. For instance, the repulsion costs constructed in \cite{martel} have the property that $\cd(x;y_0,y_1)$ converges to $-\infty$ when $x\to y_1$, and converges to $+\infty$ when $x\to y_0$, and all other level sets $\cd(-;y_0,y_1)^{-1}(s) \subset X$ are topologically connected and separating $X$ into two components, for every $s\in \bR$,

\begin{prop}\label{lipsi}
Let $c: X\times Y \to \bR$ be a cost satisfying Assumptions (A0)--(A5). Let $\psi: Y\to \bR \cup \{-\infty\}$ be a $c$-concave potential. Then $\psi$ is locally Lipschitz on its domain $\psi: dom(\psi) \to \bR$. 
\end{prop}

Recall the definition of semiconvexity \cite[Definition 10.10, pp.228]{Vil1}: 
\begin{dfn}[Semiconvexity]
A function $\phi: X \to \bR \cup \{+\infty\}$ is semiconvex on an open subset $U$ of $X$ with modulus $C>0$ at $x_0\in X$ if for every constant-speed geodesic path $\gamma(t)$, for $0\leq t\leq 1$ whose image is included in $U$, the inequality 
\begin{equation}\phi(\gamma(t)) \leq (1-t) \phi(\gamma(0)) +t \phi(\gamma(1)) +t(1-t) C dist(\gamma(0), \gamma(1))^2 \label{scineq} 
\end{equation} is satisfied for $0\leq t \leq 1$.
The function is locally semiconvex if $\phi$ is semiconvex at every $x_0\in U$, with respect to a modulus $C>0$ depending uniformly on $\gamma(0)$, $\gamma(1)$ varying in compact subsets $K$ of $U$. 
\end{dfn}

\begin{lem}\label{lip2}
Let $c: X\times Y \to \bR$ be cost satisfying Assumptions (A0)--(A2). Then every $c$-convex potential $\psi^c: X\to \bR \cup \{+\infty\}$ is $\mathscr{H}_X$-almost everywhere locally-Lipschitz on its domain $dom(\psi^c) \subset X$. Furthermore every $c$-convex potential is locally-semiconvex on $dom(\psi^c)$. 
\end{lem}
\begin{proof}
Omitted.
\end{proof}

%Under Assumptions (A0)--(A2) the graph of $\del^c \psi^c$ is a closed subset of $X\times Y$, as the following Lemma shows. 

%\begin{lem}\label{sublsc}
%Let $c: X\times Y \to \bR$ be cost satisfying Assumptions (A0)--(A2). Let $\phi=\psi^c$ be a $c$-convex potential on $X$. Then the $c$-subdifferential $\del^c \phi(x)$ is lower semicontinuous with respect to $x\in dom(\phi)$. So if $x_1, x_2, \ldots$ is a sequence in $dom(\phi)$ converging to $x_\infty \in dom(\phi)$, then the Gromov-Hausdorff limit $\lim_{k\to +\infty} \del^c \phi(x_k)$ is contained in $\del^c \phi(x_\infty)$.
%\end{lem}
%\begin{proof}
%Lemma \ref{psilem1} implies $\phi$ and $\psi=\phi^c$ are lower semicontinuous and upper semicontinuous, respectively. Let $(x_k, y_k)$ be a sequence in $dom(\phi)\times dom(\psi)$ with $y_k\in \del^c \phi(x_k)$ for $k=1,2,\ldots$. Then $-\phi(x_k)+\psi(y_k)=c(x_k,y_k)$ for all $k$. But semicontinuity implies $$\liminf_{k\to +\infty}\phi(x_k)\geq \phi(x_\infty), ~~\limsup_{k\to +\infty}\psi(y_k)\leq \psi(y_\infty).$$ Therefore $-\phi(x_\infty) + \psi(y_\infty) \geq c(x_\infty, y_\infty),$ which implies $y_\infty \in \del^c \phi(x_\infty)$, as desired.

%Lemma \ref{lip2} says $\phi$ is locally semiconvex.
%\end{proof}

%Recall our convention that $X$ is a manifold-with-corners, equipped with a riemannian distance $d$, and riemannian volume form $dx$. 

\begin{prop}\label{lip3} 
Let $c$ be cost satisfying Assumptions (A0)--(A2). Then $c$-convex potentials $\psi^c$ are $\mathscr{H}_X$-almost everywhere differentiable on $dom(\psi^c) \subset X$. Thus $dom(D \psi^c)$ is a full $\mathscr{H}_X$-measure subset of $dom(\psi^c)$.
\end{prop}
\begin{proof}
According to Lemma \ref{lip2}, the $c$-convex potentials $\psi^c$ are locally Lipschitz on their domains $dom(\psi^c) \subset X$. Rademacher's theorem, \cite[Thm 10.8, pp.222]{Vil1}, says locally Lipschitz functions are almost-everywhere differentiable on $dom(\psi^c)$ with respect to $\mathscr{H}_X$ on their domains. Therefore $\nabla_x \psi^c$ exists almost everywhere on $dom (\psi^c)$ as desired.
\end{proof}

\section{Statement of Main Results }
 The unique $c$-optimal semicouplings do not ``activate" all of the source measure whenever we have strict inequality $\int_X \sigma > \int_Y \tau$.  Equivalently the domain $dom(\psi^c)$ is a nontrivial subset of the source $X$. Informally $c$-optimal semicouplings first allocate as much as possible from low-cost regions of the source. The union of all these activated low-cost regions defines a domain designated $A \hookrightarrow X$. Specifically we define $A:=\cup_{y\in Y} \del^c \psi(y)$.

Our first Theorem \ref{A} describes a criteria to ensure that the canonical inclusion $A\hookrightarrow X$ is a homotopy-isomorphism.
\begin{thm}\label{A} 
Let $c$ be cost satisfying Assumptions (A0)--(A4), and let $\sigma$, $\tau$ be absolutely continuous with respect to $\mathscr{H}_X$, $\mathscr{H}_Y$, respectively and satisfying \eqref{abundant}. Let $\pi$ be a $c$-optimal semicoupling from $\sigma$ to $\tau$, with dual $c$-concave potential $\psi^{cc}=\psi$ , \ref{maxI}. Let $A=dom(\psi^c)$ be the active domain of the unique $c$-optimal semicoupling. Let $\beta:=\dim(Y)+2$. Suppose every $x\in X-A$ has the property that \begin{equation} \label{thmaeta}
\eta_{avg}(x):=(\mathscr{H}_Y[Y])^{-1} \int_{Y} (c(x,y)-\psi(y))^{-\beta}\cdot\nabla_x c(x,y) ~d\mathscr{H}_Y(y),  
\end{equation} is bounded away from zero, uniformly with respect to $y\in Y$. 

Then the inclusion $A \hookrightarrow X$ is a homotopy-isomorphism, and there exists strong deformation retract $X\leadsto A$. 
\end{thm}

The condition \eqref{thmaeta} is an assumption similar to (A6), and is key technical hypothesis for constructing the deformation retract. We remark that $\eta_{avg}(x)$ is uniformly bounded away from zero whenever the gradients $\nabla_x c(x,y)$, $y\in Y$, occupy a nontrivial halfspace for every $x$. Or equivalently, whenever the closed convex hull of $\nabla_x c(x,y)$, $y\in Y$, is disjoint from the origin $\textbf{0}$ in $T_x X$. Without the estimate $||\eta_{avg}(x)||\geq C >0$, our methods could only conclude that the source $X$ deformation retracts onto $\epsilon$-neighborhoods $A_\epsilon=\{x\in X| d(x,A) <\epsilon\} \hookrightarrow X$ of $A\hookrightarrow X$, where $\epsilon>0$ is a sufficiently small real number. 

%

%The applications of quadratic cost to closed convex sets $X$ with target $Y=\sE[X]$ equal to the extreme points of $X$ is limited. Indeed the active domain $A$ of optimal semicouplings is a strictly proper subset of $X$, and every point $x_0\in X-A$ in the complement fails to satisfy the hypotheses of Theorem \ref{A}. 

%For applications to source convex subsets $X$ and targets $Y=\sE[X]$, the observations of the previous paragraph motivate our introduction of anti-quadratic costs called ``repulsion" costs. As we develop our applications below, we shall verify our repulsion costs satisfy the hypotheses of Theorem \ref{A} for general semicouplings. 

%Notation. Strange terminology for variety.
Now Theorem \ref{A} is an important preliminary result. Our primary motivations are constructing retracts onto large codimension subspaces. Our next Theorem \ref{B} constructs further homotopy-reductions from the active domains $A\subset X$ to higher codimension subvarieties $\sZ$.

%Remark. The following is remark on terminology "varieties". Alludes to later results on DC subvarieties.
%We use ``subvariety" to mean a subset described by the vanishing of a collection of twice-continuously differentiable functions. In applications the functions will even be smooth. This is somewhat informal use of the term, however in [ref] we establish that the subvarieties $Z_j$, $j\geq 1$, are DC-subvarieties [ref].

%This is the descending filtration. Replace with later section.
Let $Z=Z_{\psi}: 2^Y \to 2^X$ be the singularity functor $Z(Y_I)=\cap_{y\in Y_I} \ysub$. For integers $j\geq 1$, we define $Z_j$ to be the subset of $x\in X$ where the local tangent cone is at least $j$-dimensional at $x$. The formal definition is provided in \S \ref{ksf}, c.f. \ref{dim-est}, \ref{sing-dim}. The singularity structure is naturally cellulated by the cells $$Z'(x):=Z(\del^c\psi^c(x))$$ and admits a natural filtration $$(X=:Z_0)\hh (A=:Z_1) \hh Z_2 \hh Z_3 \hh \cdots $$ of $X$ by ``subvarieties" $Z_0$, $Z_1$, etc. Under the assumptions (A0)--(A3), the cell $Z'(x)$ has well-defined tangent space $T_x Z'(x)$. Let $pr_{Z'}: T_x X \to T_x Z'(x)$ be the orthogonal projection defined relative to the Riemannian structure on $X$.

%Definition of order of zeros.
For the following construction we need introduce an auxiliary function $\nu_x(y)$ which will depend on a source point $x$ and its $c$-subdifferential $\del^c \phi(x)$ relative to a $c$-convex potential $\phi=\psi^c$. The assumptions (A0)--(A5) imply the zeros of $\psi(y_0)-\psi(y) -\cd(x,y,y_0)$ have a well-defined \emph{order} or \emph{multiplicity} as defined in the complex analysis. The order $ord(y_0)$ is a positive integer obtained by taking power series expansion of $\psi(y_0)-\psi(y) -\cd(x,y,y_0)$ with respect to $d_Y(y,y_0)$ for $y$ in a sufficiently small neighborhood of $y_0$.

\begin{dfn}\label{order}
Let $d_Y$ be the metric distance on the target $Y$. Let $\phi=\psi^c$ be a $c$-convex potential on the source $X$. For $x\in dom(\phi)$ assume $\del^c \phi(x)$ has \emph{finite} cardinality. Define \begin{equation}\label{nuorder}
\nu_x(y):=\min(~1~~,~\prod_{y_0 \in \del^c \phi(x)} d_Y(y,y_0)^{~ord(y_0)}),\end{equation} where $ord(y_0)$ is the order of the zero of $\psi(y_0)-\psi(y) -\cd(x,y,y_0)$ at $y=y_0$. 
\end{dfn}

%\begin{dfn}
%A $c$-convex potential $\phi: X \to \bR \cup \{+\infty\}$ is $\delta$-\emph{separated} if $\del^c\phi(x)$ is a $\delta$-separated discrete subset of $Y$ for every $x\in dom(\phi)$, i.e. if $d_Y(y_0,y_1)\geq \delta$ for every $y_0\neq y_1$ in $\del^c \phi(x)$. 
%\end{dfn}

\begin{dfn}
For $x\in X$, abbreviate $Y'(x):=dom(c_x)$. For a real parameter $\beta>2$ and $x\in X$, define the collection of tangent vectors $$ \eta(x, y):=|\psi(y_0)-\psi(y)+\cd(x; y, y_0)|^{-\beta} \cdot pr_{Z'}( \nabla_x \cd(x; y, y_0)). $$ For every $x$, define 
\begin{equation}
\eta_{avg}(x):=(\mathscr{H}_Y[Y'(x)])^{-1} \int_{Y'(x)} \eta(x,y) \cdot \nu^\beta_x (y) \cdot d\mathscr{H}_Y(y). 
\label{avg2}\end{equation} \`A priori, $\eta_{avg}(x)$ is a vector in $T_x Z'$. We say (UHS) Conditions are satisfied on $Z'(x)$ if $\eta_{avg}$ is bounded away from zero uniformly with respect to the source point $x$ (c.f. Definition \ref{whs}).
\end{dfn}
N.B. the parameter $\beta$ appears twice in the definition of $\eta_{avg}$: $\beta$ occurs in the definition of $\eta(x,y)$ and in the exponent of $\nu_x$. 

%The hypothesis that $\eta_{avg}(x)$ be uniformly bounded away from zero is a weak version of the following condition.
%\begin{dfn}\label{hh}
%A collection $E=\{\eta_i~|~ i\in I\} \subset T_{x} X$ of tangent vectors satisfies the \textbf{Halfspace condition} if there exists a nonzero linear functional $\ell:T_{x} X \to \bR$ with $\ell(\eta_i)>0$ simultaneously for all $i \in I$. 
%\end{dfn}
%Equivalently, Halfspace condition says the convex hull $conv[E]=conv[\{\eta_i~|~ i\in I\}] \subset T_x X$ does not contain the origin $0 \in T_x X$. 

Our next result identifies the maximal index $J\geq 1$ such that $(A=Z_1) \hh Z_{J+1}$ is a homotopy-isomorphism, and even a strong deformation retract. 
\begin{thm}\label{B} 
Let $c$ be a cost satisfying Assumptions (A0)--(A5). Suppose $\sigma, \tau$ are source, target measures absolutely continuous with respect to $\mathscr{H}_X, \mathscr{H}_Y$, respectively and satisfying \eqref{abundant}. Let $\psi^{cc}=\psi$ be a $c$-concave potential (\ref{maxI}) solving Kantorovich's dual max program, and suppose $\phi=\psi^c$ is a $c$-convex potential such that $\del^c \phi(x)$ has finite cardinality for every $x$. Let $j\geq 1$ be an integer. 

\textbf{(i)} Suppose there exists a parameter $\beta\geq 2$ such that $\eta_{avg}(x')$  is bounded away from zero, uniformly with respect to $x' \in Z_j-Z_{j+1}$; and 

\textbf{(ii)} Suppose the intersection $\sub \cap Z_{j+1} \neq \emptyset$ is nonempty for every $x'\in Z_j-Z_{j+1}$. 

Under the above hypotheses, the inclusion $Z_{j+1} \hookrightarrow Z_j$ is a homotopy-isomorphism, even a strong deformation retract. Furthermore if $J\geq 1$ is the maximal integer such that every $x'\in Z_J-Z_{J+1}$ satisfies conditions (i)--(ii), then the inclusion $Z_{J+1} \hookrightarrow Z_1$ is a homotopy-isomorphism, even a strong deformation retract.
\end{thm}

Theorem \ref{A} has the following application to quadratic costs \begin{equation}\label{bqc} b(x,y):=d(x,y)^2/2=||x-y||^2/2\end{equation} for closed subsets $X,Y$ in a Euclidean space $\bR^N$. The gradients $\{\nabla_x c(x_0,y)\}_{y \in Y}$ are a subset of $T_{x_0} X$. Observe that the closed convex hull $conv\{\nabla_x c(x_0,y)\}_{y \in Y}$ contains the origin $\textbf{0}$ in $T_{x_0} X$ if and only if $x_0$ lies in the convex hull $conv(Y)$ of $Y$ in $\bR^N$. If the activated domain $A$ of a $d^2/2$-optimal semicoupling contains the convex hull $conv(Y) \subset A$, then Theorem \ref{A} implies the inclusion $A\subset X$ is a homotopy-isomorphism. The proof of Theorem \ref{A} constructs an explicit strong deformation retract of $X$ onto $A$.  We find $A$ contains $conv(Y)$ whenever the ratio $\int_X \sigma / \int_Y \tau>1$ is sufficiently close to $1^+$, i.e. whenever the active domain $A$ is a sufficiently large subset of $X$. 

We illustrate with an example inspired by \cite{SturmH}. Let $\sigma=\mathscr{L}$ be a Lebesgue measure on $\bR^N$, and let $\tau=(100)^{-1}\sum_{i=1}^{100} \delta_{y_i}$ be an empirical Poisson sample, some radomnormalized sum of Dirac masses on $\bR^N$. Evidently \eqref{abundant} is satisfied. Consider the restriction of $b$ \eqref{bqc} to $\bR^2 \times Y$. There exists unique $b$-optimal semicoupling from $\sigma$ to $\tau$, and let $A\subset \bR^2$ be the active domain. The active domain $A$ is a union of possibly overlapping Euclidean balls. The non-active domain $\bR^N-A$ is an unbounded open subset of $\bR^N$. Under the hypotheses of \ref{A}, we define an averaged potential $f_{avg}: \bR^N -A \to \bR$ which has property that both $$f_{avg}(x_k)\to +\infty \text{~~and~~}||\nabla_x f_{avg}(x_k)|| \to +\infty $$ whenever $x_k$ is a sequence in $\bR^N-A$ converging to $\lim_k x_k=x_\infty \in \del A$. See \eqref{favg} for definition of $f_{avg}$. Now the hypotheses of \ref{A} require $f_{avg}$ have no finite critical points on the open subset $\bR^N-A$. But the nonexistence of critical points can be achieved by a simple observation: if $A\supset conv(Y)$, then for every $x\in \bR^N-A$ the gradients $\nabla_x c(x,y)$, $y\in Y$ occupy a nontrivial Halfspace of $T_x \bR^N$ with definite lengths, and this implies $\nabla_x f_{avg}$ is uniformly bounded away from zero.

The hypotheses of \ref{A} are never satisfied when we restrict $b$ \eqref{bqc} to a convex subset $X:=F$ and its boundary $Y:=\del F$. If $\sigma=1_F \mathscr{L}$, and if the target $\tau$ is supported on $\del F$, and if \eqref{abundant} is satisfied with strict inequality, then the active domain $A$ of the unique $b$-optimal semicoupling from $\sigma$ to $\tau$ will not contain $conv(Y)$. Consequently $\eta_{avg}$ will vanish somewhere on $X-A$, and the hypotheses of \ref{A} are violated. In \cite[Ch.4]{martel} a repulsion cost $c|\tau$ is defined which will satisfy (UHS) conditions throughout the non active domains and satisfy hypotheses of \ref{A}. 

\begin{figure}
\centering
\includegraphics[width=0.8\textwidth]{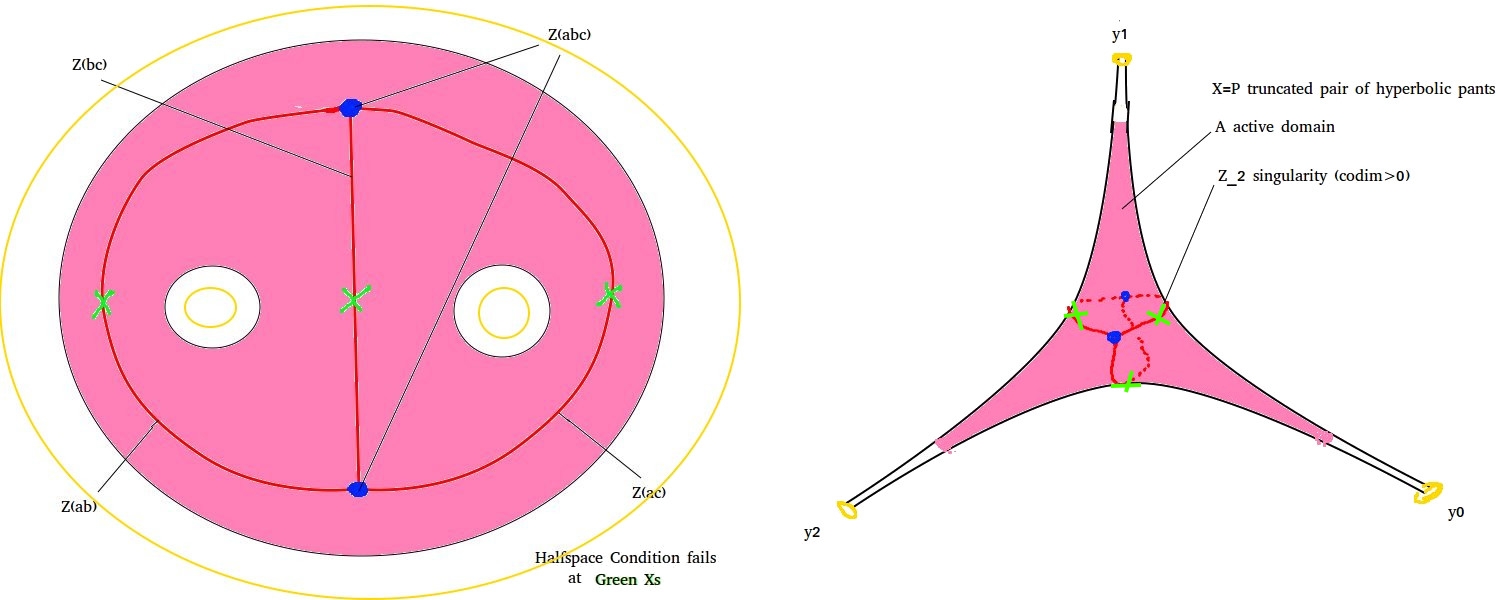}
\caption{Horospherically truncated pair of pants, with active domain relative to a repulsion cost \cite[Ch.4]{martel}. We find (UHS) conditions fail at the green points on $Z_2$. Theorems \ref{A}, \ref{B}, yield homotopy-reductions $P \leadsto Z_1 \leadsto Z_2$, where $P$ is the pair of pants. But $Z_2$ does not deformation retract to $Z_3$}
\end{figure}

\section{Averaged Gradients} 
Suppose we have a Radon measure $\nu$ on the target $Y$, absolutely continuous with respect to $\mathscr{H}_Y$, and we take the $\nu$-average of a $\nu$-family of potentials $f$. The following lemma formally establishes that this averaged-gradient is indeed the gradient field of a continuously differentiable ``averaged" potential. 

\begin{lem}\label{avgpot} 
Let $\beta\geq 2$. Let $\nu_1, \nu_2, \nu_3, \ldots$ be a sequence of empirical probability measures, i.e. renormalized sums of Dirac masses, which converge as $N\to +\infty$ in the weak-$*$ topology to the renormalized probability measure $(\bar{\nu}[Y])^{-1} \cdot \bar{\nu}$ on $Y$. Then:

(i) For $x\in X-A$, the limit \begin{equation}\label{limit}\lim_{N\to +\infty} (1-\beta)^{-1}~ \int_Y (c(x,y) -\psi(y))^{1-\beta} d\nu_N(y)\end{equation} exists and converges to the finite integral \begin{equation}\label{favg} f_{avg}(x):=(1-\beta)^{-1}~\frac{1}{\bar{\nu}[Y]} \int_Y (c(x,y)-\psi(y))^{1-\beta}. d\bar{\nu}(y).\end{equation} 

(ii) The rule $f_{avg}: X-A \to \bR$ defines a continuously differentiable function with gradient $$\nabla_x f_{avg}=(\bar{\nu}[Y])^{-1} \int _Y \nabla_x (c(x,y)-\psi(y))^{-\beta}.d\bar{\nu}(y).$$ 
\end{lem}

\begin{proof}
If $c$ satisfies (A0)--(A4), then the limit defining $f_{avg}$ converges uniformly on compact subsets of $X-A$. So the limit \eqref{limit} exists and is finite. Moreover the uniform convergence on compact subsets implies (ii), since the approximants are continously differentiable on $X-A$. Therefore $\nabla_x f_{avg}$ is the average of $\nabla_x (c(x,y)-\psi(y))^{-\beta}$ with respect to $\bar{\nu}$, as desired.  
\end{proof}

%BELOW IS NOT TRUE WITHOUT THE MAX_LOWER BOUND.
%Observe $f_{avg}(x)$ diverges to $-\infty$ when $x\in X-A$ converges to $A$. Furthermore $\nabla_x f$ diverges to the point-at-infinity $\infty$ when $x\in X-A$ converges to $A$. 

%Next we fix $y_0\in Y$, abbreviating $t_0=t(y_0)$ and $f_0(x):=(c(x,y_0)-t_0)^{1/2}$. 

%Suppose $\{x_j\}$ is sequence in $X-A$ which converges to $x_\infty \in \{f_0=0\} \subset A$. Note $\{f_0=0\}=\{x~|~c(x,y_0) \leq t_0\}$. We want to compare the asymptotics of $f_{avg}$ and $f_0$ in neighborhoods of $\{f_0=0\}$. 

%It is elementary to upperbound the averaged gradient \begin{equation}||\nabla_x f_{avg}|| \leq (\bar{\nu}[Y])^{-1} \int_Y ||\nabla_x f_y||d\bar{\nu}(y),  \label{upper}\end{equation} using triangle-inequality. The inequality \eqref{upper} says the averaged gradient diverges only if there is a measurable subset of $Y$ where the integrand diverges. Our applications require the converse: the divergence of at least one gradient $\nabla_x f_y$ should imply the divergence of the average gradient. This converse motivates our Assumption (A5). Informally we require that there is small cancellation between the terms $\{\nabla_x f_y~|~ y \in dom(c_{x})\}$. Then the ``large" gradients $\nabla_x f_y$ will have ``large" average with respect to $\bar{\nu}$. Here is formal definition.

For $y\in Y$, $x\in X-A$, and $\beta\geq 2$, we abbreviate 
\begin{equation}\label{fy} f_y(x):=\frac{1}{1-\beta}~(c(x,y)-\psi(y))^{1-\beta} \end{equation} and define \begin{equation}\label{favgdef}
f_{avg}: X-A \to \bR, ~~~f_{avg}(x)=(\nu[Y])^{-1} \int_Y f_y(x).d\nu(y).
\end{equation}

\begin{dfn}[\textbf{Property (C)}] \label{a5spec}
The collection of functions $\{f_y|~y\in Y\}$ satisfies Property (C) throughout $X-A$ with respect to the uniform probability measure $\bar{\nu}:=\frac{1}{\nu[Y]}~\nu$ if there exists constant $C>0$ such that \begin{equation}
||\nabla_x f_{avg}|| \geq C \int_Y ||\nabla_x f_y||.d\bar{\nu}(y) \label{maxineq}
\end{equation}
pointwise throughout $X-A$.
\end{dfn}

When $Y$ is finite, $\#(Y)<+\infty$, the estimate \eqref{maxineq} requires the ratio $$||\nabla_x f_{avg}||/ \max_{y\in Y}||\nabla_x f_y||$$ be uniformly bounded away from zero throughout $X-A$. In this case, Property (C) and (UHS) conditions ensures the divergence of the average $\nabla_x f_{avg}$ whenever a gradient summand $\nabla_x f_y$ diverges. 

When $Y$ is infinite with positive dimension, the pointwise divergence of an integrand $||f_y||\to +\infty$ need not imply the divergence of the average $f_{avg}$ for every choice of $\beta\geq 2$ -- rather the rate at which $f_y$ diverges must be sufficiently large. When cost $c$ satisfies Assumptions (A0)--(A5), Lemma \ref{lipsi} proves that $c$-concave potentials $\psi: Y\to \bR \cup \{-\infty\}$ are locally Lipschitz throughout $dom(\psi)$. This implies $\beta =\dim(Y)+2$ is sufficient. If $\{x_k\}_k$ is a countable sequence in $X-A$, then $f_{avg}(x_k)$ diverges to $+\infty$ if and only if $f_y(x_k)$ diverges for $y$ belonging to some subset $V\subset Y$.

\begin{lem}\label{conc} 
Let $c: X\times Y \to \bR$ be cost satisfying (A0)--(A5), as above, and $A\subset X$ the active domain of a $c$-optimal semicoupling. Fix $y_0\in Y$, and abbreviate $f_0(x):=f_{y_0}(x)=\frac{1}{1-\beta}(c(x,y_0)-\psi(y_0))^{1-\beta}$ for some $\beta\geq 2$, $x\in X$. 
Suppose:

\textbf{(a)} $\nabla_x c(x,y_0)$ is uniformly bounded away from the origin; and

\textbf{(b)}  $\nabla^2_{xx} c(x,y_0)$ is uniformly bounded above with respect to $x\in X-A$.

Then for every $K>0$, there exists $\epsilon>0$ such that $\nabla_{xx}^2 f_0\geq K.Id>0$ in the direction of $-\nabla_x c(x,y_0)$ throughout the $\epsilon$-neighborhood of $\{f_0=+\infty\}$ in $X-A$.

\end{lem}
\begin{proof}
The function $f_0$ is well-defined on $\{c(x,y_0)>\psi(y_0)\} \subset X$. If $x$ converges to $x_\infty \in \{c(x,y_0) \leq \psi(y_0)\}$, then both $f_0$, $\nabla_x f_0$ diverge to infinity.  We find $\nabla^2_{xx} f_0$ is equal to $$-\beta (c(x,y_0)-\psi(y_0))^{-1-\beta} \nabla_x c(x,y_0) \otimes \nabla_x c(x,y_0)+(c(x,y_0)-\psi(y_0))^{-\beta} \nabla^2_{xx} c(x,y_0).$$ By Proposition \ref{a5spec} the gradients $\nabla_x c(x,y_0)$ are uniformly bounded away from zero in neighborhoods of $\{f_0=+\infty\}$. Assumption (A2) implies $\nabla^2_{xx} c(x,y_0)$ is uniformly bounded above on superlevel sets $\{f_0 \geq T\}$ for all $T>0$. Factoring out the term $(c(x,y_0)-\psi(y_0))^{-\beta}$, we find $\nabla_{xx}^2 f_0$ is positively proportional to \begin{equation}\label{concest} -\beta~(c(x,y_0)-\psi(y_0))^{-1}~\nabla_x c(x,y_0) \otimes \nabla_x c(x,y_0) + \nabla_{xx}^2 c(x,y_0).\end{equation} We observe \eqref{concest} diverges to $-\infty$ when $(c(x,y_0)-\psi(y_0))^{-1}$ diverges to $+\infty$. This implies $\nabla^2_{xx} f_0$ is negative semidefinite when $c(x,y_0)-\psi(y_0)>0$ is sufficiently small, and strongly convex $$\nabla_{xx}^2 f_0 \geq K >0$$ in the direction of $-\nabla_x c(x,y_0)$. N.B. the change of sign. 

By Assumption (A1) the sublevels $\{x ~|~\psi(y_0)\leq c(x,y_0) \leq \psi(y_0)+\epsilon'\}$ are compact subsets of $X-A$ for every $\epsilon'>0$. This implies a sufficiently small $\epsilon>0$ exists for which $\nabla^2_{xx} f_0\geq K.Id>0$ throughout the $\epsilon$-neighborhood of $\{f_0=+\infty\}$ in the direction $-\nabla_x c(x,y_0)$. 
\end{proof}

%Remark. We want the hessian to be convex in the direction of the flow $\eta(x,avg)$. This means the flow will asymptotically contract and become distance-decreasing and lipschitz, etc.. In above lemma, we need distinguish between the directions of $\nabla_x c(x,y_0)$ and the directions of $\nabla_x f_0$, which are opposite directions! This is because $f=1/c$ essentially. Hence we increase $f$ in the direction where $c$ is decreasing, i.e. negative gradient $\nabla_x c$. This caused initial confusion when reviewing the above proof.

\begin{exx} 
To illustrate Lemma \ref{conc}, consider the $f(x)=x^{-\beta}$ for $x\geq 0$, $\beta>0$. The gradient flow $x'= (-\beta)x^{-1-\beta}$ is bounded away from zero in neighborhoods of the pole at $x=0$, and indeed diverges to $+\infty$. Moreover $f''(x)$ is obviously bounded away zero and diverging to $+\infty$ as $x\to 0^+$. 

For initial condition $x_0>0$, the integral curve of the negative gradient flow is equal to $x(s)=(x_0^{2+\beta} - \beta(\beta+2)s)^{1/(2+\beta)}$, which converges in finite time to the pole at $x=0$ over the interval $0\leq s\leq \frac{1}{\beta(\beta+2)}x_0^{2+\beta}$. Thus we find $\omega(x_0)=\frac{1}{\beta(\beta+2)}x_0^{2+\beta}$ varies continuously with respect to $x_0>0$, and is even Lipschitz. Compare Lemma \ref{maxinterval} below. 

\end{exx}

\section{Finite-Time Blow-Up}

The blow-up in finite time is typical property of the gradient flow defined by the potentials $f_0$. Actually our applications require verifying these same properties for the averaged potential $f_{avg}$ and its gradient $\nabla_x f_{avg}$. 

\begin{lem}[Asymptotic Convexity of Average Gradient Flow]\label{conc1} 
Let $f_{avg}$ be the average defined in Lemma \ref{avgpot}, equation \eqref{favg}, with exponent $\beta=\dim(Y)+2$. If the distance from $x\in X-A$ to the boundary $\del A$ is sufficiently small, then $f_{avg}$ is strongly convex in the direction of $\nabla_x f_{avg}$.
\end{lem}
\begin{proof}
Let ${x_k}$ be a sequence in $X-A$ converging to a point $x_\infty\in \{f_{avg}=+\infty\}$. The choice of $\beta$ says $f_{avg}$ diverges if and only integrands $f_y$ diverge, and there exists a subset $V \subset Y$ such that $f_y$ diverges to $+\infty$ for every $y\in V$. The divergence of $f_y, y\in V$ also implies the divergence of the gradients $\nabla_x f_y$ and Hessians $D_{xx}^2f_y$, (see \eqref{concest} in proof of \ref{conc}). Moreover the Hessians $D^2 f_{avg}|_{x_k}$ are positive semidefinite when $k$ is sufficiently large, being the asymptotic to the average rank-one quadratic forms $\langle \nabla_x f_y , - \rangle^2$. So $D^2f_{avg}[\nabla_x f_{avg}]$ is asymptotic to the average of \begin{equation} \label{innprod}
\langle \nabla_x f_y, \nabla_x f_{avg} \rangle^2
\end{equation} for $y\in V$. Now we claim $$\lim_{k\to +\infty} \langle \nabla_x f_y|_{x_k}, \nabla_x f_{avg}|_{x_k} \rangle^2=+\infty, \text{~unless~} \nabla_x f_y,  \nabla_x f_{avg} \text{~are orthogonal}.$$ The (UHS) conditions imply $\nabla_x f_{avg}$ is uniformly bounded away from zero, and therefore the inner products \eqref{innprod} are not identically zero for all $y\in V$. This implies the divergence of $D^2f_{avg}[\nabla_x f_{avg}]|_{x_k}$ as $k\to +\infty$. 

%This claim is established by observing 
%Next for $k$ sufficiently large, we claim the gradients $\{\nabla_x f_y|_{x_k} ~|~ y\in V\}$ satisfy Halfspace Conditions. Indeed the boundary $\del A$ is semiconvex as seen from $X-A$, c.f. \cite[\S 5, pp.697]{McCaf}. 
%If $x_k$ is sufficiently close to $\cap_{y\in V} \{f_y=+\infty\}$, then $D^2 f_{avg}$ is asymptotic to the truncated average $ \nu[Y]^{-1} \int_V D^2 f_y d\nu(y)$. With Lemma \ref{conc}, this implies $D^2 f_{avg}$ is strongly convex in the directions of $\nabla_x f_{avg}$ for $x$ sufficiently close to $\del A$. 
\end{proof}

\begin{lem}[Finite-time Blow-up] \label{ftbu}
Suppose the functions $\{f_y~|~ y\in Y\}$ satisfy \eqref{maxineq} as above. Then for every initial value $x_0\in dom(f_{avg})$, the gradient flow defined by the average gradient $x'=\nabla_x f_{avg}$ diverges to infinity in finite time.
\end{lem}
\begin{proof}
The estimate \eqref{maxineq} shows the gradient $\nabla_x f_{avg}$ is uniformly bounded away from zero in the neighborhoods of the poles $\{f_{avg}=+\infty\}$ in $X-A$. Moreover $f_{avg}$ is asymptotically convex in neighborhoods of the poles using Lemma \ref{conc1} along directions of $\nabla_x f_{avg}$. This implies the gradient flow $x'=\nabla_x f_{avg}$, $x(0)=x_0$ blows-up in finite-time for every initial value $x_0\in X-A$. 
\end{proof}

Informally the estimate \eqref{maxineq} implies every step in the discretized gradient flow (e.g., Euler scheme) has a definite size. Meanwhile the asymptotic convexity of Lemma \ref{conc} implies the discretized gradient flow well approximates the continuous gradient flow. But if step-sizes have a definite magnitude, then we definitely approach the poles after a finite number of steps and the integral curves blow-up in finite time.

\section{Continuity of $\omega$}
The blow-up in finite time (Lemma \ref{ftbu}) implies the maximal forward-time interval of existence for the gradient flow is a bounded interval $I(x_0):=[0, \omega(x_0)) \subset [0, +\infty)$. For general ordinary differential equations, it's known that $\omega(x_0)$ is lower semicontinuous as a function of $x_0$: for a sequence of initial values $x_0, x_1, \ldots$ converging to some $x_\infty$, we have $\omega(x_\infty) \leq \liminf_{k\to +\infty} \omega(x_k)$. See \cite[Theorem 2.1, pp.94]{Hart}. In our particular setting, it is further necessary to establish the continuity of this maximal interval of existence. This is established in Lemma \ref{maxinterval} below.

\begin{lem}[Continuity of Max Interval of Existence]\label{maxinterval}
%Assume the cost $c$ satisfies Assumptions (A0)--(A5).
Let $c$ satisfy Assumptions (A0)--(A5). Then the maximal intervals of existence $I(x_0)=[0, \omega(x_0))$ of solutions to the initial value problem \ref{ode} vary continuously with respect to the initial point $x_0$. In otherwords, $x_0\mapsto \omega(x_0)$ varies continuously with $x\in X-A$.
\end{lem}
\begin{proof}
Our assumptions imply the domain $X-A$ is a complete open set. Furthermore $\eta_{avg}(x)$ being uniformly bounded away from zero implies the trajectories $s\mapsto \Psi(x_0, s)$ are finite on compact subsets of $X-A$. Therefore $\omega(x_0)$ is characterized by the two limits $$\lim_{s\to \omega(x_0)^-} \Psi(x_0, s) \in \del A ,~~~ \lim_{s\to \omega(x_0)^-} ||\Psi(x_0, s)||=+\infty,$$ and here we take advantage of the divergence $||\eta_{avg}(x)||\to +\infty$ as $x\to \del A$. 

Moreover the asymptotic concavity (Proposition \ref{conc1}) of the average potential defining $\eta_{avg}(x)$ implies the flow defined by \ref{ode} is asymptotically contracting. This implies $\omega(x_0)$ is actually a Lipschitz function of $x_0$, i.e. satisfying $|\omega(x_1)-\omega(x_0)|\leq C.||x_1-x_0||$ for some constant $C=C(x_0)$ depending on $x_0\in X-A$. Hence $\omega$ is a continuous function, as desired.  
\end{proof}

%Hypothesis: \omega(x_0) is Lipschitz function of x_0. 

%For $x_0\in X-A$, the unique trajectory $s\mapsto \Psi(x_0, s)$ satisfying \ref{ode} over the maximal interval $I(x_0)$ has either $\lim_{s\to \omega(x_0)^-} ||\Psi(x_0, s)||=\infty$ or $\lim_{s\to \omega(x_0)^-} \Psi(x_0, s) \in \del A$. 

\section{Proof of Theorem \ref{A}}

\begin{proof}
[Proof of Theorem \ref{A}] 
The active domain $A$ can be expressed as $\cup_{y \in Y} \{x~|~ c(x,y) \leq \psi(y)\}$ for the Kantorovich potential $\psi: Y \to \bR$. If (UHS) condition is satisfied at $x\in X - A$, then the $\nu$-average $\eta_{avg}(x)=\nabla_x f_{avg}(x)$ gradients $\{\nabla_x c(x,y) ~|~ y \in Y\}$ of the gradients is nonzero and uniformly bounded away from zero. We divide the argument into two cases. 

(Case I) Assume $Y$ is finite with $N=\#(Y) <+\infty$.
Define \begin{equation}
\eta_{avg}(x):=N^{-1} \sum_{y \in Y}[(c(x,y)-\psi(y))^{-\beta}\cdot\nabla_x c(x,y)].  \label{avgI}
\end{equation}  Evidently when $Y$ is finite, the sum \eqref{avgI} is finite vector. Then $x\mapsto \eta_{avg}(x)$ is a well-defined nonvanishing vector field on $X-A$ which diverges whenever a denominator converges $c(x,y) \to \psi(y)^+$. Thus $\eta(x, avg)$ is finite if and only if $x\in X-A$. 

We propose integrating the vector field $x\mapsto \eta_{avg}(x)$ throughout $X-A$ to obtain the desired retraction. For every initial point $x_0 \in X-A$, there exists a unique solution $\Psi(x_0, s)$ to the ordinary differential equation 
\begin{equation} 
\frac{d}{ds}|_{(x_0,s)} \Psi=\eta_{avg}(\Psi(x_0,s)), ~~~ \Psi(x_0, 0)=x_0,   \label{ode}
\end{equation} and defined over a maximal interval of existence $I(x_0)=[0, \omega(x_0))$. According to Lemma \ref{maxinterval} the maximal interval $I(x_0)$ varies continuously with respect to $x_0 \in X-A$. Moreover orbits $\{\Psi(x_0,s)~|~ s\in I(x_0)\}$ converge in finite-time to the boundary $\del(X-A)=\del A$ for every initial value $x_0\in X-A$. Lemma \ref{maxinterval} also proves the orbits can be continuously reparameterized to obtain a continuous mapping $\Psi': (X-A) \times [0,1] \to X$ defined by \begin{equation}\label{reparam} \Psi'(x_0,s)=\Psi(x_0, s \omega(x_0)).\end{equation}

%Thus we obtain a 1-parameter family of diffeomorphisms $$\Psi: (X-A) \times [0,+\infty) \to X-A,$$ where $\Psi(x,s)$ is the unique solution of the
%ordinary differential equation 
%\begin{equation} 
%\frac{d}{ds}|_{(x',s')} \Psi=-\eta(\Psi(x',s'),avg), ~~~ \Psi(x', 0)=x'   \label{ode}
%\end{equation} for all
%$(x',s') \in (X-A) \times [0, +\infty)$.

%Finite-Time BlowUp! Causes problems.

%Thus we obtain a $1$-parameter family of diffeomorphisms
%$$\Psi: (X-A) \times [0,+\infty) \to X-A,$$ where $\Psi(x,s)$ is the unique solution of the
%ordinary differential equation 
%\begin{equation} 
%\frac{d}{ds}|_{(x',s')} \Psi=-\eta(\Psi(x',s'),avg), ~~~ \Psi(x', 0)=x'   \label{ode}
%\end{equation} for all
%$(x',s') \in (X-A) \times [0, +\infty)$. 

%According to Lemma \ref{ftbu}, the flow $\Psi$ defined by the gradient $\eta(x,avg)$ blows-up in finite-time, and therefore every initial value $x_0 \in X-A$ determines a unique maximal interval of existence $I(x_0)$. 

%The standard results on continuous dependance on initial conditions and parameters proves the convergence-time, and the maximum interval of existence $I(x_0)$, varies continuously with respect to $x_0$. 

(Case II) Suppose $Y$ is infinite set, with uniform measure $\mathscr{H}_Y$. We define $\eta_{avg}(x)$ according to the vector-valued Bochner integral \begin{equation}
\eta_{avg}(x):=(\int_{dom(c_{x})} d\mathscr{H}_Y)^{-1} \int_{dom(c_{x})} (c(x,y)-\psi(y))^{-\beta}\cdot\nabla_x c(x,y) \nu^{~\beta}_x (y)d\mathscr{H}_Y(y),  \label{avgII} \end{equation} where $dom(c_x)$ is closed compact subset of $Y$ for every $x\in X$ by Assumption (A0). Assumption (A5) implies the exponent $\beta=\dim(Y)+2$ is suitable according to Proposition \ref{lipsi}. Lemma \ref{avgpot} implies the vector field $\eta_{avg}(x)=\nabla_x f_{avg}$ is the gradient of a continuously differentiable potential $f_{avg}$ defined on $X-A$.

The proof proceeds as in (Case I). The vector field $x\mapsto \eta_{avg}(x)$ is well-defined nonvanishing vector field on $X-A$ which diverges to $+\infty$ whenever some denominator converges $c(x,y) \to \psi(y)^+$. We integrate the gradient fields and obtain the retraction of $X-A$ onto the poles $\del A$. The flow converges in finite-time by Lemma \ref{ftbu}. We reparameterize the flow according to equation \eqref{reparam}, and obtain a continuous deformation using Lemma \ref{maxinterval}.
\end{proof}

\section{Kantorovich's Contravariant Singularity Functor}\label{ks-1}
%Our thesis proposes a bridge between measure and algebraic topology. The bridge is realized by  contravariant functors $Z: 2^Y \to 2^X$, where $2^Y$ designates the category of closed subsets $Y_I$ of $Y$, and where morphisms are the inclusions $Y_I \subset Y_J$ between closed subsets $Y_I$, $Y_J$ whenever they exist. Likewise for $2^X$. 

The singularity functor $Z: 2^Y \to 2^X$ is defined relative to a cost $c$ on $X \times Y$ satisfying Assumptions (A0)--(A4). Let $\sigma, \tau$ be absolutely continuous with respect to the Hausdorff measures $\mathscr{H}_X$, $\mathscr{H}_Y$ on $X, Y$ respectively. By Monge-Kantorovich duality \ref{maxI} there exists unique $c$-minimizing measures and $c$-concave potentials $\psi^{cc}=\psi$ on $Y$. The $c$-subdifferential $\del^c \psi$ is uniquely determined.

\begin{dfn}[Kantorovich Singularity]
\label{ks} 
Let $\psi: Y\to \bR \cup \{-\infty\}$ be a $c$-concave potential on $Y$. The Kantorovich functor $Z: 2^Y \to 2^X$ is defined by the rule 
$$Y_I\mapsto Z(Y_I):=\cap_{y\in Y_I} \ysub.$$ We declare $Z(\emptyset_Y):=X$ for the empty subset $\emptyset_Y$ of $Y$.
\end{dfn}

The definition $Z(Y_I)=\cap_{Y_I} Z(y)$ yields a \textit{contravariant functor} $Z:2^Y \to 2^X$ where the inclusions $Y_I \hookrightarrow Y_J$ correspond to reverse inclusions $Z(Y_I) \hookleftarrow Z(Y_J)$ in $X$. The contravariant functor $Z$ is uniquely prescribed by the choice of $(c, \sigma, \tau)$ whenever $c$ satisfies (A0)--(A4), and $\sigma, \tau$ are absolutely continuous as defined above. For $(c, \sigma, \tau)$ satisfying the assumptions above, the singularity $Z=Z(c, \sigma, \tau)$ will generally admit many closed subsets $Y_I \subset Y$ for which the cells $Z(Y_I)\subset X$ are empty $Z(Y_I)=\emptyset_X$. It is useful to restrict ourselves to the nontrivial image of $Z$ and formally define the support. 

\begin{dfn}
The support of the contravariant functor $Z: 2^Y \to 2^X$ is the subcategory of $2^Y$, denoted $spt(Z)$, whose objects are the closed subsets $Y_I$ of $Y$ for which $Z(Y_I)$ is nonempty subset of $X$. 
\end{dfn}
So $spt(Z)=\{Y_I \subset Y|~ Z(Y_I) \neq \emptyset_X\}.$ 
Note that $\emptyset_Y \subset Y$ is object in subcategory $spt(Z)$, since $Z(\emptyset_Y)=X$ according to Definition \ref{ks}. 

%Definition: relative neighborhood. cellular neighborhood.

For given $x\in X$, we interpret $Z'(x):=Z(\del^c \psi^c (x))$ as a ``cellular neighborhood" of $x$ in the active domain $A=:Z_1$. This cellular neighborhood is not necessarily an open subset of the source $X$. The $c$-concavity $\psi^{cc}=\psi$ provides explicit equations describing $Z'(x)$. See Section \ref{locequations} and \eqref{dceq}.

\section{Local Topology and Local Dimensions of $Z$}\label{locequations}
The present section examines the local differential topology and dimensions of Kantorovich's contravariant functor $Z: 2^Y \to 2^X$, where the cost $c$ satisfies Assumptions (A0)--(A5). If $\psi:Y\to \bR \cup \{-\infty\}$ is a $c$-concave potential, then we first seek a description of the differential topology of the cells $$Z'(x):=Z(\del^c \psi^c(x)).$$ The main results are \ref{dim-est} and \ref{sing-dim}. The proof of \ref{dim-est} is essentially a local adaptation of the recent work of \cite{KitMc}, and \ref{sing-dim} is a corollary to a theorem of Alberti \cite{Alberti}. %We are grateful to Prof. R. J. McCann for many useful conversations. 

Let the reader recall the definitions of $c$-concavity \ref{c-concavity} and $c$-subdifferentials \ref{subdif}. The $c$-concavity $\psi^{cc}=\psi$ represents a pointwise inequality on $Y$, namely \begin{equation}\label{flineq}-\psi^c(x)+\psi(y) \leq c(x,y) \text{~~for all~} (x,y)\in X\times Y.\end{equation} The case of equality $-\psi^c(x)+\psi(y) = c(x,y)$ is most important, and occurs if and only if $y\in \del^c \psi^c (x)$, and if and only if $x \in \del^c \psi(y)$. We recall that $c$-optimal semicouplings $\pi$ are supported on the graphs of $c$-subdifferentials,  hence the equality $-\psi^c(x)+\psi(y)=c(x,y)$ holds $\pi$-a.e., when $\psi$ is a $c$-concave maximizer to Kantorovich's dual program.  

Recall that if $\phi: \bR^N \to \bR \cup \{+\infty\}$ is a convex lower semicontinuous function on some $N$-dimensional vector space, then the following are equivalent \cite[I.5.3, pp.23]{ET}: 

- $\phi$ is continuous and finite at $x\in \bR^N$, with $\del \phi(x)$ a singleton; 

- $\phi$ is differentiable at $x$ with $\del \phi(x)=\{D\phi(x)\}$.

Thus $\phi$ is continuous throughout the interior of $dom(\phi)$. So if $\phi=\psi^c$ is a $c$-convex potential on $X$, then $\del^c \phi(x)$ is multivalued if and only if $\phi$ is non-differentiable at $x$. 

Returning to $c$-convexity: in case $\del^c \phi(x_0)$ is a singleton, say $\{y_0\}$, then $-\psi^c(x_0)+\psi(y)<c(x_0,y)$ for every $y\neq y_0$. The set of $x$'s with $y_0\in \del^c\phi(x)$ is then characterized as the subset where $$\psi(y_0) \geq c(x,y_0)+\phi(x)$$ according to \eqref{flineq}, i.e. as the superlevel set of $x\mapsto c(x,y_0)+\phi(x)$ and Assumptions (A0)--(A3) imply these superlevel sets are closed. More concretely, we find $Z(\{y_0\})$ is characterized by the equations \begin{equation} Z(\{y_0\}=\{x \in X|~\psi(y_0)=c(x,y_0)+\phi(x)\}.\label{singlesub} \end{equation} Equivalently,  we find $x\in \del^c \psi(y_0)$ if and only if \begin{equation} y_0\in argmax[\{ \psi(y) - c(x,y)+c(x,y_0)~|~ y\in Y \}]. \label{yy}\end{equation}

If $x$ is such that $\del^c \psi^c(x)$ is not a singleton, say including two distinct points $y_0,y_1$, then $x$ satisfies the equations $$\psi^c(x)+c(x,y_0)=\psi(y_0), ~~~\psi^c(x)+c(x,y_1)=\psi(y_1).$$ Subtracting these two equations, we can eliminate $\phi(x)$ and obtain the equation $$0=\psi(y_1)-\psi(y_0)+\cd(x,y_0,y_1),$$ where we abbreviate $c_\Delta(x;y, y'):=c(x,y)-c(x,y')$ for the two-pointed cross difference. If $y_0\in \del^c \psi^c(x)$, and $y\notin \del^c \psi^c(x)$, then $$0<\psi(y)-\psi(y_0)+\cd(x,y_0,y).$$ %For arbitrary $y,y'$ the quantity $\psi(y)-\psi(y')-\cd(x,y,y')$ has arbitrary sign.

%If $x_0\in X$ is such that $\del^c \psi^c(x_0)$ is a singleton, say $\{y_0\}$, then $-\psi^c(x_0)+\psi(y)<c(x_0,y)$ for every $y\neq y_0$. 

% If $\psi^c$ is differentiable at $x\in X$, then $\del^c \psi^c(x)$ is a singleton, say $\{y_0\}$. 

%For $x_0\in X$, abbreviate $Z'(x_0):=Z(\del^c \psi^c(x_0))$. We interpret $Z'(x_0)$ as a local cell containing $x_0$ in the activated domain $A$ of the $c$-optimal semicoupling. 

\begin{dfn}
Let $\phi$ be a $c$-convex potential on $X$, $\phi^{cc}=\phi$. We say $x\in dom(\phi)$ is a singular point if $\phi$ is not differentiable at $x$.  
\end{dfn}
If $x_0\in X$ is a singular point, then $\del^c \phi(x_0)$ is not a singleton and the cell $Z'(x_0)$ is described by the system of equations \begin{equation} Z'(x_0)=\{x\in X~|~ 0=\psi(y_0)-\psi(y_1) +\cd(x; y_1, y_0), ~ y_1 \in \del^c \psi^c(x_0), y_1\neq y_0\} \label{dceq}\end{equation} according to \eqref{yy}, where $\psi=\phi^c$.

From the equations \eqref{dceq} we obtain the following:
\begin{lem}\label{ldc} 
Under Assumptions (A0)--(A3), the cell $Z'(x)=Z(\del^c \psi^c(x))$ is a closed locally DC-subvariety in $X$ for every singular point $x\in dom(\psi^c)$. 
\end{lem}
\begin{proof}
The cell $Z'(x)$ is the intersection of sets of the form $\del^c \psi(y)$, which are closed \ref{psilem1}. Assumption (A1) implies $x\mapsto \nabla^2_{xx} c(x,y)$ is locally bounded above on $X$, uniformly in $y$. Therefore every $x$ admits a neighborhood $U$ and a constant $C\geq 0$ such that $\nabla_{xx}^2 c(x,y) \leq C.Id$ uniformly with respect to $y$ throughout $U$. This implies $x\mapsto c(x,y)$ is locally semiconcave function on $X$, uniformly in $y$. Therefore the cross-differences $x \mapsto \cd(x;y,y')$ are locally DC-functions, uniformly in $y,y'$ in $Y$. This observation and equation \eqref{dceq} implies $Z'(x)$ is locally-DC subvariety of $X$.  
\end{proof}

To illustrate, consider the Euclidean quadratic cost $c(x,y)=||x-y||^2/2$ on $X=\bR^N$. The cross-difference $\cd(x,y,y')$ is an affine function of $x$, namely $$\cd(x,y,y')=\langle x, y'-y\rangle + ||y||^2/2-||y'||^2/2.$$ Therefore the cells $Z'(x)$ are locally affine subsets of $\bR^N$ for every singular point $x$.

We caution the reader that Lemma \ref{ldc} does not the specify the dimension of $Z'(x)$ -- the dimension of these cells is addressed below \ref{dim-est}. Again $Z'(x)$ is viewed as a local cell on the active domain $x\in Z_1=A$. Our next step is to describe the space of directions $T_x Z'(x)$, compare \cite[Def.10.4.6, pp.257]{Vil1}. The definition of the vector field $\eta_{avg}(x) \in T_x Z'$ involves the orthogonal projection $pr_{Z'}:T_x X \to T_x Z'$ of $T_xX$ onto the space of directions.

\begin{lem}\label{orthog}
Let $x\in X$ be a singular point where $\del^c \psi^c(x)$ is not a singleton.  Then the space of directions $T_x Z'(x)$  is a subset of the orthogonal complement  $$orthog[\{\nabla_x \cd(x; y_0, y_1)~ |~ y_1,y_0 \in \del^c \psi^c (x) \}]$$ in $ T_x X.$ 
\end{lem}
\begin{proof}
Let $\phi=\psi^c$. For $y_0, y_1 \in Y$, we abbreviate $A(x,y_0,y_1)=\cd(x,y,y_0)-\psi(y)+\psi(y_0)$. If $x$ is a singular point, then $Z'(x)$ is characterized by the vanishing $A(\bar{x}, y, y')=0$ for $y,y'\in \del^c \phi(x)$. A first-order deformation $\eta$ at $x$ and tangent to $Z'(x)$ must preserve the system of equations $\{A(x,y,y')=0 | ~~ y, y' \in \del^c \phi (x)\}$. But this only if $\eta \in T_x X$ satisfies the homogeneous linear equations $\eta \cdot \nabla_x \cd(x, y, y')=0$ for every $y, y'\in \del^c \phi(x)$. 
\end{proof}

%CLAIM: in the covariant notation, take the differentials $dA(x,y,y_0)$, and their kernel. This is the space-of-directions of $Z'(x) at x.

%But the deformation $\eta$ at $x$ may also have effect of rapidly increasing the relative cross-differences centred at $y$ near $y_0$. 

% Comment: why is above description not good enough for hausdorff dimension of Z'(x)? 

% Under Assumptions (A0)--(A4) we confirm the above dimension estimate, 
%Using a nonsmooth implicit function theorem adapted from \cite[Thm 3.8]{KitMc}, we describe $Z'(x) \cap B(x, \epsilon)$ as the graph of a bilipschitz DC-function [reference equation].

%[G: [localsplitting] --> ] [statement of theorem].

%START HERE.

Lemma \ref{orthog} indicates the \emph{expected} Hausdorff dimension of $Z'(x)$, namely the dimension of the orthogonal complement $\{\nabla_x \cd(x,y,y')~|~ y,y' \in \del^c \psi^c(x)\}$. 
A standard application of Clarke's Nonsmooth Implicit Function theorem confirms this dimension estimate. Some preliminary notation is convenient. Recall $X$ is a Riemannian manifold-with-corners, equipped with Riemannian exponential mapping $exp_x:T_x X \to X$. If $(X,d)$ is complete Cartan-Hadamard manifold, then $exp_x$ is diffeomorphism between $T_x X$ and the universal covering space $\tilde{X}$. On general Riemannian manifold-with-corners $X$ the exponential map is a local diffeomorphism between sufficiently small open neighborhoods $U$ of $x_0 \in X$ with open balls $U'$ in the tangent space $T_{x_0} X$. So we have $C^1$ diffeomorphisms between local neighborhoods of $x_0 \in X$ with neighborhoods of $0$ in euclidean space $\bR^n$, with $n=\dim(X)=\dim(T_{x_0} X)$. Thus for every $x_0 \in X$ there exists a local diffeomorphism splitting $B_\epsilon(x_0)$ as a product $B_\epsilon(x_0') \times B_\epsilon(x_0'')$, viewed as subset of $\bR^{n-j} \times R^j$ with $x_0=(x_0', x_0'')$ for $j\geq 0$. 
%$$F: orthog(span\{\nabla_x \cd\})) \to span\{\nabla_x \cd\}=T_x X.$$ 

%This exhibits $Z'(x)$ locally at $x'$ as the codimension-$j$ lipschitz DC-subvariety $Z'(x) \cap B(x', \epsilon) \approx graph(F)$ with $j$ defined as above.
Recall $c$ satisfies Assumptions (A0)--(A4). Let $\psi:Y\to \bR \cup \{-\infty\}$ be a $c$-concave potential, with $\phi=\psi^c$. Let $Z_1:=A$ be the active domain of the semicoupling defined by $\psi$, and let $x_0\in Z_1$. 
\begin{prop}
\label{dim-est} 
Under the above hypotheses, let $x_0$ be a singular point. Suppose the gradients $\{ \nabla_x \cd(x;y_0, y_1) ~|~y_0,y_1 \in \del^c \psi^c(x_0) \}$ span a $j$-dimensional subspace of $T_{x}X$ for every $x \in Z'(x_0)$. Then $Z'(x_0)$ is a codimension-$j$ local DC-subvariety of the active domain $Z_1$.

%where $$j:= \dim[span\{\nabla_x \cd(x; y_0, y_1) ~|~ y_0,y_1 \in \del^c \psi^c (x_0)\}].$$  
\end{prop}

%When the gradients $\{ \nabla_x \cd(x';y_0, y_1)~|~ y_0,y_1 \in \del^c \psi^c(x)\}$ satisfy Halfspace Conditions and Definition \ref{a5spec}, then \begin{equation}\label{codimZ} codim_{Z_1} Z_j = j-1 \text{~for every~} j \geq 1. \end{equation} Observe that (HS) Conditions fail whenever $Z_j$ is empty, and evidently the formula \eqref{codimZ} no longer applies. 

%Our assumptions imply the cross-differences $x \mapsto \cd(x; y,y')$ are locally lipschitz functions with respect to local coordinates on $X$ (See \ref{localdc}). The equations \eqref{dceq} imply $Z'(x)$ is characterized as level-sets of $x\mapsto \cd(x;y,y')$ restricted to the active domain $Z_1$ in $X$. 
\begin{proof} 
Fix some $y_0\in \del \phi(x_0)$. Abbreviate $A(x,y):=c(x,y)-c(x,y_0)$. Consider the mapping $G: B_\epsilon(x_0) \to \bR^d$ defined by $$G(x)= (A(x,y_1), A(x,y_2), \ldots A(x,y_j)) $$ for $y_0, y_1, \ldots, y_j \in \del^c \psi^c(x_0)$. Assume $$\{\nabla_x  A(x_0,y_i, y_0)\}_{1 \leq i \leq j}$$ is a linearly independent subset of $T_{x_0} X$. The map $G:B_\epsilon(x_0) \to \bR^j$ is local DC-function according to Lemma \ref{ldc}. 

We use exponential mapping to obtain local $C^1$-diffeomorphism between an open neighborhood of $x_0$ in $X$. Next we apply the $DC$-implicit function theorem as stated in \cite[Thm 3.8]{KitMc}, and conclude there exists $\epsilon>0$ and a biLipschitz DC-mapping $\phi$ from $B_\epsilon(x_0') \subset \bR^{n-j}$ to $B_\epsilon(x_0'') \subset \bR^j$ such that, for all $x=(x', x'') \in  B_\epsilon(x_0') \times B_\epsilon(x_0'') \subset \bR^{n-j} \times \bR^j$ we have $G(x)=G(x',x'')=0$ if and only if $x''=\phi(x')$. 

The basic subdifferential inequalities \eqref{dceq} imply $G(x)=0$ if and only if $x\in Z'(x_0) \cap B(x_0,\epsilon)$. Now because $Z'(x_0)$ can be covered by countably many sufficiently small open balls, we conclude that $Z'(x_0)$ is a local DC-subvariety with Hausdorff dimension $\dim_{\mathscr{H}} Z'(x_0)$ $=n-j$. 
\end{proof}

The above \ref{dim-est} is a type of ``constant-rank" theorem for Lipschitz maps, c.f. \cite{clarke1976inverse}.

%In a neighborhood of $x'$, we apply a nonsmooth implicit function theorem (\cite[proof of Theorem 10.50, pp.260]{Vil1}) to describe $Z'(x) \cap B(x', \epsilon)$ as the graph of a lipschitz function $$F: orthog(span\{\nabla_x \cd\})) \to span\{\nabla_x \cd\}=T_x X.$$ This exhibits $Z'(x)$ locally at $x'$ as a subset of a codimension-$j$ lipschitz subvariety $Z'(x) \cap B(x', \epsilon) \approx graph(F)$ when $j$ is defined as above.

%In otherwords, along $Z'(x)$ the equations satisfied at $x'$ are those satisfied by the cross-differences $\cd(x';y,y')$ for all $y,y'$ in the subdifferential $\del^c \psi^c (x')$. But our assumptions imply the individual level-sets of the cross-difference functions are locally codimension-one subvarieties. The (HS) conditions says there exists $\alpha>0$ such that $||v|| \geq \alpha >0$ for all $v\in conv(\{\nabla_x\cd\}) \subset T_x X$. But this implies [INCOMPLETE].

%The hypothesis on $\nabla_x \cd$ implies there exists a nonzero direction $v$ such that $\langle v, u\rangle \geq \alpha >0$ for all $u \in orthog(span)$. 
%[INCOMPLETE]
%The (HS) conditions on the gradients $\nabla_x \cd$ says: 
%The proof of the Nonsmooth Implicit Function theorem in \cite[Theorem 10.50, pp.260]{Vil1}, when adapted to our hypotheses, proves the
%Consequence of Lemma \ref{orthog}.

N.B. The description of $T_{x} Z'(x)$ is symmetric with respect to $y_1, y_0$. There is further symmetry from the additive relations between cross-costs $$\cd(x;y_0,y_1)+\cd(x;y_1, y_2)=\cd(x;y_0, y_2).$$ This implies the obvious estimate $ codim Z'(x) \leq \#(\del^c \psi^c (x))-1$ under general conditions.

%If the domain $dom \psi^c=A$ is open in $X$, then Proposition \ref{dim-est} implies every cell $Z'(x)$ is lipschitz subvariety with codimension $\geq \dim(orthog(span))$. 

%As $x$ varies over the domain $dom \psi^c$, 

%The proof of \ref{dim-est} does not necessarily require the DC-implicit function theorem, since the cross-differences $x\mapsto \cd(x,y,y')$ are assumed to be $C^2$ according to Assumption (A1). We could therefore have applied a standard $C^1$- or $C^2$-implicit function theorem (e.g. \cite[Thm. 2.12]{spivak}) to conclude $Z'(x)$ is a $C^1$ (or $C^2$) subvariety of $X$. However we prefer the $DC$ perspective, which appears the more natural category for optimal transport methods. 

Recently \cite{KitMc} obtained an explicit parameterization of singularities arising from Euclidean quadratic costs, employing a hypothesis of affine independence between the connected components of the subdifferentials $\del^c \psi^c(x)$. Their parameterization requires a global splitting $X=X_0 \times X_1$ of the source domain to express singularities (``tears" in their terminology) as the graphs of $DC$-functions $G: X_0 \to X_1$ as above. Further hypotheses on the convexity of source and target are required for their arguments. 

%The argument of Proposition \ref{dim-est} is a localized version of \cite{KitMc}.

\section{The Descending Filtration $\{Z_j\}_{j\geq 0}$}\label{ksf}
Following an idea of Prof. D. Bar-Natan \cite{Bar2002} we ``skewer the cube" $2^X$ (more accurately the support $spt(Z)$) according to local codimensions and obtain a descending filtration. The previous section described various cells $Z'(x)$ which decompose the activated source $A$ of optimal semicouplings. Assembling these cells into the Kantorovich functor $Z: 2^Y \to 2^X$ leads to a useful filtration of the source $X$. 

\begin{dfn}\label{dimdef}
For integers $j=0,1,2,\ldots$, let $$Z_{j+1}:=\{x \in X | ~~  \dim[span\{\nabla_x \cd(x; y, y') ~|~ y,y' \in \del^c \psi^c (x)\}] \geq j\}$$ where $Z'(x)=Z(\del^c \psi^c (x))$ for a $c$-concave potential $\psi^{cc}=\psi$ on target space $Y$.
\end{dfn}
According to the definition, $Z_j$ is supported on the subcategory $spt(Z)$ of $2^Y$ for every $j\geq 1$. The activated support $A$ coincides with the support of $Z_1$. The cells $Z(Y_I)$ are closed (\ref{psilem1}). If the support $spt(Z)$ of the functor is finite, then the unions $Z_{j+1}$ are closed. However if $spt(Z)$ is infinite, then the union $Z_2$, $Z_3$, $\ldots$ is possibly not closed and depends on the cost $c$. For example, the Euclidean quadratic cost $c=d^2/2$ typically has $Z_2$ not closed when the target $Y$ has $\dim(Y)>0$. This motivates our study of repulsion costs in \cite[Ch.4]{martel}. 

%For instance, if $c=d^2/2$ is the Euclidean quadratic cost, then

%Our definition exhibits $Z_j$ as a union of manifolds-with-corners $Z'(x)$.

%\begin{lem}
%Let $c$ be a cost satisfying Assumptions (A0)--(A4), and $Z=Z(\sigma, \tau, c)$ the corresponding singularity functor. Then for every $j\geq 1$, we find $Z_j$ is a closed topological subset of $X$.
%\end{lem}
%\begin{proof}
%If $x_1, x_2, x_3, \ldots$ is countable sequence in $Z_j$, converging to a limit $\lim_{k\to +\infty} x_k=x_\infty$ in $X$, then we claim $x_\infty \in Z_j$. Let $\phi$ be a $c$-convex potential solving Kantorovich's dual maximization program \ref{maxI}. Now the $c$-subdifferential $x\mapsto \del^c \psi(x)$ is lower semicontinuous with respect to $x$, according to Lemma \ref{sublsc}. This implies $\del^c \psi^c(x_\infty)$ is a subset of $Y$ which contains the Gromov-Hausdorff limit $\lim_{k\to +\infty}\del^c \psi^c(x_k)$.  But then Definition \ref{dimdef} implies $x_\infty \in Z_j$.

%\end{proof}

%So we filtrate the singularity according to codimension, and obtain a descending chain of closed subsets $$(X=Z_0)\hh (A=Z_1) \hh Z_2 \hh Z_3 \hh \cdots  $$ of the source $X$. At this point 
Now we replace the $c$-subdifferentials with an important ``localized" version using (Twist) condition (A4). The $c$-subdifferential $\del^c \psi^c$ of a $c$-convex function $\psi^c : X\to \bR \cup \{+\infty\}$ is non-local subset of $Y$. The subsets $\del^c \psi^c \subset Y$ depends on global data, namely the values of $c(x,y)$ for all $y\in dom(c_{x})$. Hence $\del^c \psi^c (x_0)$ depends on the values of $\psi(y)$ and $\psi^c(x)$ for every $y\in Y, x\in X$ and not simply the local behaviour of $\psi^c $ near $x_0$. It is useful to introduce a local subdifferential, namely the so-called ``subgradients" of a function $\psi^c:X\to \bR \cup \{+\infty\}$.

\begin{dfn}[Local Subgradients $\del_\bullet \phi$]
Let $U$ be open set in $X$, sufficiently small such $U$ is $C^1$-diffeomorphic to an open subset of $\bR^n$ where $n=\dim(X)$. Let $\phi: U \to \bR$ be a function. Then $\phi$ is subdifferentiable at $x$ with subgradient $v^*\in T_x^*X$ if $$\phi(z)\geq \phi(x) + v^*(z-x) + o(|z-x|) \text{~for all~} z \text{~near~}x.$$ Let $\del_\bullet \phi(x) \subset T_x^* X$ denote the set of all subgradients to $\phi$ at $x$. Here $o$ is the ``little-oh" notation.
\end{dfn} 
Evidently the subgradient $\del_\bullet \phi(x)$ is local, and depending on the values of $\phi$ near $x$. Moreover $\del_\bullet \phi(x)$ is a closed convex subset of $T_x^*X$ for every $x\in dom(\phi)$. %The definition of $\del_\bullet \phi$ coincides with the so-called ``Clarke subdifferential", c.f. \cite[10.49]{Vil1}.

%When $\phi=\psi^c$ is $c$-convex potential, the Assumptions (A0)--(A3) imply $\psi^c$ is locally semiconvex function over its domain $dom(\psi^c) \subset X$. 
The Assumptions (A0)--(A3) imply $\psi^c: X\to \bR \cup \{+\infty\}$ is locally semiconvex over its domain $dom(\psi^c) \subset X$ (Lemma \ref{lip2}). In otherwords, every $x\in dom(\psi^c)$ admits an open neighborhood $U$ of $x$ and a constant $C\geq 0$ such that $D^2 \psi^c \geq -C Id$ throughout $U$, where $D^2\psi^c=D^2_{xx} \psi^c$ is the distributional Hessian in the source variable $x$ (\cite[Theorem 14.1, pp.363]{Vil1}). So given a $c$-convex potential $\psi^c$ on $X$, for every $x_0\in X$ there exists open neighborhood $U$ and $C\geq 0$ such that $\psi^c |_{U_x} + C ||x||^2/2$ is strictly convex throughout $U$. This implies the local subdifferential and subgradients $\del_\bullet \psi^c(x)$ coincide with the convex-analytic subdifferential of the local function $\psi^c|_U+C||x||^2/2$, when restricted to the sufficiently small neighborhood $x$ in $X$. 

%According to Lemma[ss], the $c$-convex potentials $\psi^c$ on $X$ are locally semiconvex under Assumptions (A0)--(A3). 

Having introduced the local subdifferential, there is important comparison between $\del^c \psi^c (x_0) \subset Y$ and $\del_\bullet \psi^c(x_0)$, assuming (Twist) condition Assumption (A4). This relation is the inclusion
\begin{equation} \del^c \psi^c (x_0) \subset \{y\in Y | -\nabla_x c(x_0,y) \in \del_\bullet \psi^c (x_0)\}. \label{grad}\end{equation} 

Observe $\del^c \psi^c (x_0)$ is nonempty whenever $x_0\in dom(\psi^c)$. We abuse notation and  denote $\nabla_x c(x,y)$ for the canonical covector in $T_x^* X$, with the tangent vector in $T_xX$ using the ambient Riemannian structure. The inclusion \eqref{grad} allows us to replace the global $c$-subdifferential $\del^c \psi^c$ with the local convex set of subgradients $\del_\bullet \psi^c$. The inclusion is generally strict. However it produces basic upper bounds on the Hausdorff dimension of $\del^c \psi^c$. We quote the following theorem of G. Alberti: 
\begin{thm}[\cite{Alberti}]\label{alb}
Let $f: \bR^n \to \bR$ be proper lowersemicontinuous convex function. For every $0<k<n$, let $S^k(f) \subset \bR^n$ be the subset defined by $$S^k(f):=\{x\in \bR^n| \dim_{ \mathscr{H}} (\del_\bullet (f)) \geq k\}.$$ Then $S^k(f)$ can be covered by countably many $(n-k)$-dimensional DC-manifolds.
\end{thm} 
In otherwords $S^k(f)$ is countably $(n-k)$-rectifiable and has Hausdorff dimension $\leq (n-k)$. If $f$ is $+\infty$-valued on $\bR^n$, then Alberti's method shows $S^k(f)$ can be covered by countably many $(n'-k)$-dimensional manifolds where $n'=\dim_{\mathscr{H}}( dom(f))$. The domain $dom(f)$ is a closed convex subset of $\bR^n$ having well-defined Hausdorff dimension.

%Suppose cost $c$ satisfies Assumptions (A0)--(A3). 

\begin{prop}
\label{sing-dim}
Let $c: X\times Y \to \bR \cup \{+\infty\}$ be cost satisfying Assumptions (A0)--(A4). Let $\psi: Y\to \bR \cup \{-\infty\}$ be $c$-concave potential $\psi^{cc}=\psi$ with singularity functor $Z: 2^Y \to 2^X$. Define $n':=\dim_{\mathscr{H}}(dom (\psi^c))$. Then for every integer $j\geq 1$, the subvariety $Z_j$ has Hausdorff dimension $$\dim_{\mathscr{H}} (Z_{j}) \leq n'-j+1.$$ 
\end{prop}
\begin{proof}

Consider the inclusion \eqref{grad}. For given $x_0$, let $U$ be the open neighborhood and $C>0$ such that $\psi^c |_{U} + C ||x||^2/2$ is strictly convex throughout $U$. For $x\in U$ the $c$-subdifferentials $\del^c \psi^c(x)$ are contained in the closed convex local subdifferentials  $\del_\bullet \psi^c|U(x)$. Now we have equality \begin{equation}\del_\bullet (\psi^c|_U(x) + C |x|^2/2) = \del_\bullet \psi^c|_U(x) + C \langle -, x\rangle \label{subadd}.\end{equation} So the local subdifferential $\del_\bullet (\psi^c_U + C |x|^2/2)$ in $T_x^* X$ is an affine translate of $\del_\bullet \psi^c|_U$ by the linear functional $C\langle - ,x\rangle$. 

Next we apply Alberti's theorem to the localized convex function $\psi^c|_U + C |x|^2/2$. Thus $S^k(\psi^c|_U + C|x|^2/2)$ and $S^k(\psi^c|_U)$ can be covered by countably many $(n'-k)$-dimensional manifolds where $n':=\dim_{\mathscr{H}}(dom \psi^c|_U)$.

Finally we relate $\dim_{\mathscr{H}} (\del_\bullet \psi^c|_U(x))$ to the Hausdorff dimension of $Z'(x)=Z(\del^c \psi^c(x))$ $=\cap_{y\in \del^c \psi^c (x)} \del^c \psi(y)$. From the definition of subgradients, we have $$conv(\{ \nabla_x c(x,y) | y\in \del^c \psi^c(x)\}) \subset \del_\bullet \psi^c(x).$$ Moreover the closed convex hull $conv(\{ \nabla_x c(x,y) | y\in \del^c \psi^c(x)\})$ has dimension $$j=\dim(span\{\nabla_x \cd(x,y,y_0)~|~ y\in \del^c \psi^c(x) \}).$$ For every $x\in Z_j$ we conclude $Z_j \cap U \subset S^j(\psi^c|_U + C |x|^2/2),$ which according to Alberti's theorem \ref{alb} yields the upper bound $ \dim_{\mathscr{H}}(Z_j \cap U) \leq n'-j.$ To conclude, we observe that $Z_j$ can be covered by countably many open neighborhoods $U$ of points $x\in Z_j$. 

%This follows from the fact that $Z_j$ is topologically closed subset of $X$, and therefore every covering by open sets can be reduced to a countable covering.  
\end{proof}

We supplement \ref{sing-dim} with the following application of Clarke's Implicit Function Theorem, which gives a criterion for the singularities $Z_2, Z_3, \ldots$ to be \emph{closed} subsets of $X$.
 
%\begin{dfn}
%Let $I$ be a finite subset of nonzero vectors $I \subset \bR^N-\{\textbf{0}\}$. The barycentre fibre of $I$ is the set $R$ of nontrivial $\bR_{\geq 0}$-linear dependance relations of $I$; that is $R$ consists of all $\alpha \in \bR_{\geq 0}^I-\{\textbf{0}\}$ such that $\sum_{v\in I} \alpha(v) v =\textbf{0}$, and such that $\alpha(v)\geq 0$ for every $v\in I$, and $\alpha\neq 0$.    
%\end{dfn}

%Evidently $R$ is a closed convex subset contained in the octant $\bR^I_{\geq 0} \hookrightarrow \bR^I$ and disjoint from the origin $\textbf{0}$. If $I$ is linearly independant, then $R$ is empty. 

%\begin{lem}
%Let $I \subset \bR^N-\{\textbf{0}\}$ be finite subset with barycentre fibre $R$. Let $F=conv(I)$ be the closed convex hull of $F$. Then $$\dim(F)=\#(I)-1-\dim(R).$$
%\end{lem}
%\begin{proof}
%Consequence of rank-nullity theorem from linear algebra.
%\end{proof}

\begin{prop} 
\label{closedZ}
Let $c:X\times Y \to \bR$ be a cost satisfying Assumptions (A0)--(A4). Let $\phi=\psi^c$ be a $c$-convex potential on $X$. Suppose the sup-norm of the gradients $\nabla_x \cd(x,y_0, y_1)$, $y_0, y_1\in \del^c \phi(x)$ is nonzero, uniformly with respect to $x\in Z_2$. Then $Z_2$ is a closed subset of $X$. 

%For $x\in X$, and integer $j\geq 0$, let $r_{x,j}$ be the supremum of the $j\times j$-minors of the matrix with columns $\nabla_x \cd(x,y_0,y_1)$ for $y_0, y_1\in \del^c \phi(x)$. If $r_{x,j}$ is bounded away from zero, uniformly with respect to $x\in Z_j$, then $Z_j$ is a closed subset of $X$.

%For $x\in X$, let $K_x:=\{conv(\nabla_x\cd(x,y_0,y_1) ~|~y_0, y_1 \in \del^c \psi^c(x)\}) \subset T_x X$ be the closed convex hull. If $j\geq 2$ is an integer such that the convex sets $K_x$ are disjoint from the origin $\textbf{0}$ in $T_x Z'(x)$, uniformly with respect to $x\in Z_j$, then $Z_j$ is a closed subset of $X$. 
\end{prop}
\begin{proof}
Consequence of Clarke's Implicit Function Theorem, c.f. \cite[Proof of Theorem 10.50, pp.262--264]{Vil1}.
\end{proof}

The above proposition \ref{closedZ} can be generalized to the following criterion for $Z_j$ to be a closed subset. If $I$ is a collection of vectors, then $span_\bZ(I)$ is the group of all finite $\bZ$-linear combinations $\sum_{v\in I} n_v v $, for $n_v\in \bZ$.  We define the ``height" of $I$, denoted $ht(I)$, to be the \emph{volume} of the quotient $span(I)/span_{\bZ}(I)$ with respect to a fixed Lebesgue measure.

\begin{prop}
Let $c:X\times Y \to \bR$ be a cost satisfying Assumptions (A0)--(A4). Let $\phi=\psi^c$ be a $c$-convex potential on $X$. Fix $j\geq 1$. For every $x\in Z_j$, let $I_x$ be the collection of gradients $$I_x:=\{\nabla_x \cd(x,y_0, y_1)~|~y_0, y_1 \in \del^c \phi(x)\}.$$ Suppose the heights $ht(I_x)$ are bounded away from zero, uniformly with respect to $x\in Z_j$. Then $Z_j$ is a closed subset of $X$. 
\end{prop}
\begin{proof}
We leave the proof to the reader, being again a consequence of Clarke's Implicit Function Theorem. 
\end{proof}

%\begin{cor}
%Suppose the hypotheses of Prop. [ref] are satisfied for all $x\in Z_j$, for some $j\geq 2$. Then $Z_j$ is a closed subset of $X$.
%\end{cor}

%This implies $\dim T_xZ'(x) \leq \dim_{\mathscr{H}}(\del_\bullet \psi^c(x))$ according to Lemma \eqref{orthog}. 
%ERROR ABOVE? 
%Evidently \eqref{grad} implies 
%$Z'(x) \subset \cap_{y\in \del^c \psi^c(x)}$
%[Singularity set]. Definition. FILTRATION. 
%INCOMPLETE.
%Our argument in \ref{dim-est} and \ref{dim2} confirms this Hausdorff dimension bound for optimal semicouplings with respect to costs satisfying our Assumptions (A0)--(A5). 
%[INCOMPLETE]

%One must exercise caution because the subvarieties $Z_j$ are not necessarily closed subsets of $X$. The main technical tool is an appropriate Nonsmooth Implicit Function Theorem, say \cite[Thm 10.50, pp.261]{Vil1}. 

\section{Proof of Theorem \ref{B}}\label{zz}
%Section \ref{22retract} of Chapter \ref{22} constructs a continuous deformation retract of the source $X$ onto an activated domain $A$, whenever the Condition (C) and Halfspace Conditions were satisfied (see Definition \ref{a5spec} and Theorem \ref{r1}). 

This section generalizes the homotopy reduction constructed in \ref{retract1} and the proof of \ref{A}, and establishes a retraction procedure $Z_1 \leadsto Z_2 \leadsto \cdots \leadsto Z_{J+1},$ defined up to some maximal index $J \geq 1$ for which the inclusions $$Z_1 \hh Z_2 \hh \cdots Z_{J+1}$$ are simultaneously homotopy isomorphisms. These retractions require our cost $c$ satisfy Assumptions (A0)--(A5), a specific form of (A6), the $c$-subdifferentials have finite cardinality and the Uniform Halfspace condition \ref{whs}. The retract is obtained by integrating a vector field, denoted $\eta_{avg}(x)$, whose flow under the above assumptions yields continuous deformation retracts (Theorem \ref{retract2}). %Our motivation for these retracts arose from attempting to interpret Soul\'e-Ash's ``well-rounded retract" \cite{Ash1}, \cite{Soule} in the category of semicouplings and Kantorovich duality. Applications of the well-rounded retract in the ``geometry-of-numbers"  can be found in \cite{SP}, \cite[\S A.6.4]{Stein}. 

An important hypothesis for constructing our retractions is the (UHS) condition. Some notation is necessary. Abbreviate $Y'(x):=dom(c_x)$. Let $Z=Z(c, \sigma, \tau)$ be the corresponding Kantorovich functor. For $x\in Z'$, let $pr_{Z'}: T_{x} X \to T_{x} Z'$ denote the orthogonal-projection mapping, where $T_{x} Z'$ is the space of directions of $Z'$ at $x$.

%A\' priori, the space of directions $T_x Z'$ is only a cone in $T_x X$, c.f. \cite[Def.10.4.6, pp.257]{Vil1}. 

\begin{dfn}
\label{whs}

Let $\phi=\psi^c$ be a $c$-convex potential. For $x_0\in dom(\psi^c)$ abbreviate $Z':=Z(\del^c \psi^c(x_0))$. Select any $y_0\in \del^c \psi^c(x_0)$. Recall the definition of $\nu_x(y)$ from \ref{order}. For parameter $\beta \geq 2$, $x\in Z'$, $y\in Y$, define tangent vectors $\eta(x,y) \in T_x X$ by the equation \begin{equation}\eta(x, y):=|\psi(y_0)-\psi(y)+\cd(x; y, y_0)|^{-\beta} \cdot pr_{Z'}( \nabla_x \cd(x; y, y_0)). \label{lower0}
\end{equation} 

We say Uniform Halfspace (UHS) conditions are satisfied at $x$ in $Z'$ with respect to the parameter $\beta$ if: 

\textbf{(UHS1)} the Bochner integral $\eta_{avg}(x)$ defined by \begin{equation}\eta_{avg}(x):=(\mathscr{H}_Y[Y])^{-1} \int_{Y} \eta(x,y) . \nu^{~\beta}_x(y) . d\mathscr{H}_Y(y) \label{lowera}
\end{equation} is nonzero finite vector in $T_{x}Z'-\{\textbf{0}\}$; and 

\textbf{(UHS2)} there exists a constant $C>0$, uniform with $x\in Z'$, for which the estimate \begin{equation} 
||\eta_{avg}(x)|| \geq C \int_{Y} ||\eta(x,y) . \nu^{~\beta}_x(y)|| d\mathscr{H}_Y(y) >0 \label{lowerb}
\end{equation} holds.
\end{dfn}

We make some remarks. First, we observe (UHS2) basically implies (UHS1). Second, the definition of $\eta(x,y)$ is independent of the choice of $y_0 \in \sub$. Third, in practice the parameter $\beta>2$ is taken sufficiently large to ensure the divergence of $\eta_{avg}(x)$ whenever $x$ converges into $\del Z'$. If $c$ satisfies Assumption (A5), then $\psi(y)$ is locally Lipschitz (Prop.\ref{lipsi}) and the exponent $\beta=\dim(Y)+2$ is sufficient. %Finally, we observe that symmetry improves (UHS) conditions, i.e. if $\Gamma$ acts isometrically on $X\times Y$, then $c$-optimal semicouplings will also be $\Gamma$-equivariant whenever $\sigma, \tau$ are equivariant.

%Condition(UHS1) requires this divergence be integrable and finite with respect to the uniform measure $\mathscr{H}_Y$ restricted to $Y'(x)-\sub$.

%The issue is the Bochner integral in equation \eqref{lowera} is an improper integral, where the integrands $\eta(x,y)$ diverge to infinity when $x$ approaches the

%is an improper integral, e.g. whenever $Y'(x')- \sub$ is a connected open subset of $Y$. Evidently $\eta(x',y_*)$ diverges when $y_* \in Y'(x')$ converges towards the closed subset $\sub$ in $Y'(x')$. 

%Thirdly the $c$-concavity of the potential $\psi^{cc}=\psi$ implies the scalar values $\psi(y_0) - \psi(y_*) + \cd(x,y_*,y_0)$ are strictly positive finite $(>0)$ for $y_0 \in \del^c \psi^c(x)$ and $y_* \in Y'(x) -\del^c \psi^c(x)$. The scalar $\psi(y_0) - \psi(y_*) + \cd(x,y_*,y_0)$ vanishes for every $y_*\in \del^c \psi^c(x), y_0 \in \del^c \psi^c(x)$. The expression $\eta(x,y_*)$ vanishes whenever $c(x,y_*)=+\infty$ (we adopt the convention $+\infty - (+\infty)=+\infty$). The expression defining $\eta(x,y_*)$ diverges to $+\infty$ for every $y_* \in \del^c \psi^c(x)$, or whenever $x$ converges in $ Z_j$ to some point $Z_{j+1} \subset Z_j$. 

Now we state our theorem.
\begin{thm}[Local Cellular Retraction 
$Z'\cap Z_j \leadsto Z' \cap Z_{j+1}$]
%Let $c:X\times Y\to \bR \cup \{+\infty\}$ be cost satisfying Assumptions (A0)--(A6). Suppose the source $\sigma$ and target measure $\tau$ are absolutely continuous with respect to the Hausdorff measures on $X$, $Y$, with $\rho=\int_X \sigma / \int_Y \tau \geq 1$. Let $\pi_{opt}$ be the $c$-optimal semicoupling from $\sigma$ to $\tau$ and $Z=Z(c,\sigma, \tau)$ be the contravariant singularity functor $Z: 2^Y \to 2^X$ (\ref{ks}). 

%Let $x \in Z_1$ be a point supported on the activated domain of an optimal semicoupling $\pi_{opt}$. 

Suppose $\del^c \psi^c(x)$ is finite for every $x\in dom(\psi^c)$. Abbreviate $Z':=Z'(x)=Z(\del^c \psi^c(x))$. If (UHS) Conditions are satisfied for all points $x' \in Z' \cap Z_j$ (\ref{whs}), and $Z' \cap Z_{j+1} \neq \emptyset$ is nonempty, then there exists a continuous map $$\Psi: (Z' \cap Z_j) \times [0,1] \to Z' \cap Z_j $$ 
such that:

(i) $\Psi(x,s)=x$ for all $x \in Z' \cap Z_{j+1}$; and

(ii) $\Psi(x,0)=x$ for all $x\in Z'$; and 

(iii) $\Psi(x'',1) \in Z' \cap Z_{j+1}$ for all $x'' \in Z' \cap Z_j$.

In addition, if the above hypotheses are satisfied, then $Z'\cap Z_{j+1}$ is a strong deformation retract of $Z' \cap Z_j$ and the inclusion $$Z'\cap Z_{j+1} \hookrightarrow Z'\cap Z_j$$ is a homotopy isomorphism.
\label{retract1} 
\end{thm}

The vector field $\eta_{avg}(x)$ defined in Theorem \ref{retract1} is tangent to the cell $Z'(x)$, and therefore the flow is constrained to the cells $Z'(x)$. For $x$ varying over $Z_j$, the vector field $\eta_{avg}(x)$ varies continuously with respect to $x$ according to Lemma \ref{psilem1}. The hypothesis $Z' \cap Z_{j+1} \neq \emptyset$ cannot be ignored, and is definitely necessary for continuity of the gradient flow of $\eta_{avg}$. 

%\begin{lem}
%Let the hypotheses of Theorem \ref{retract1} be satisfied. If $\{x_k\}$ is a sequence in $Z'\cap Z_j$ converging $\lim_{k} x_k = x_\infty$, and if $\eta_{avg}(x_k)$ diverges to infinity, then $x_\infty \in Z'\cap Z_{j+1}$.
%\end{lem}
%\begin{proof}
%Our definitions imply $\eta_{avg}(x)$ diverges only if there exists $\bar{y}$, disjoint from $\del^c \psi^c(x_k)$, for which $\del^c \psi^c(x_\infty)\supset \{\bar{y}\}\cup \del^c\psi^c(x_k)$. The presence of the orthogonal-projection $pr_{Z'}$ and the (UHS) conditions says the average projections  $pr_{Z'} (\nabla_x \cd(x_\infty, \bar{y}, y))$, $y\in \del^c\psi(x_k)$ are nonzero and diverging as $x_k\to x_\infty$.

%\end{proof}

Next we claim the mappings $$\{\Psi: Z'(x)\cap Z_j \to Z'(x) \cap Z_j\}$$ assemble to a continuous retraction $Z_j \times [0,1] \to Z_j$, for $x\in Z_j$, which establishes that $Z_j \hookleftarrow Z_{j+1}$ is a homotopy isomorphism.

\begin{thm}[Global Retraction]
Let $c, \sigma, \tau, \rho, Z, Z', \psi, \phi$ be defined as in Theorem \ref{retract1}. If $j\geq 1$ is an integer such that (UHS) Conditions are satisfied for all points $x \in Z_j$, and $Z'(x) \cap Z_{j+1} \neq \emptyset$ for every $x\in Z_j-Z_{j+1}$, then the local homotopy equivalences $$\{\Psi: (Z'\cap Z_j) \times [0,1] \to Z'\cap Z_{j+1}\}$$ constructed in Theorem \ref{retract1} assemble to a continuous global deformation retract $Z_j \leadsto Z_{j+1}.$ 

In addition, if $J\geq 1$ is that maximal integer where (UHS) conditions are satisfied throughout $Z_J$, then composing the retractions $\{\Psi\}$ produces a codimension-$J$ homotopy isomorphism $Z_1 \simeq Z_{J+1}.$ 
\label{retract2}
\end{thm}
The proof of \ref{retract2} is analogous to our proof of Theorem \ref{retract1}, which is the ``base case" retraction of $X=Z_0$ onto the activated domain $A=Z_1$. The formal proof of \ref{retract1} required Lemmas \ref{ftbu}, \ref{maxinterval}. Likewise the formal proofs of \ref{retract2} require analogous lemmas, summarized in:

\begin{lem}\label{L;ft}
In the notations of \ref{retract1}--\ref{retract2} be satisfied. Let $j\geq 1$ be such that (UHS) conditions are satisfied for all $x\in Z_j$, and $Z'(x)\cap Z_{j+1}\neq \emptyset$ for every $x\in Z_j - Z_{j+1}$, and $\del^c\psi^c$ is has pointwise finite cardinality. 

For $x_0 \in Z_j - Z_{j+1}$, consider the initial value problem \begin{equation}\label{avgflow}
x'=\eta_{avg}(x), ~~x(0)=x_0.
\end{equation}

(a) For every initial value $x_0\in Z_j - Z_{j+1}$, the flow $\Psi$ defined by \eqref{avgflow} diverges to infinity in finite time.

(b) The maximum interval of existence $I(x_0):=[0, \omega(x_0))$ of solutions to \eqref{avgflow} varies continuously with respect to the initial value $x_0\in Z_j-Z_{j+1}$.
\end{lem}

\begin{proof}
We follow the arguments of \ref{ftbu}. First (UHS) implies the flow \eqref{avgflow} is well-defined and extendable throughout the interior of the cells $Z'(x_0)\subset Z_j$. At the $\omega$-limit point $$\bar{x}:=\lim_{t\to \omega(x_0)^-} x(t)$$ we claim $\bar{x}\in Z'(x_0)\cap Z_{j+1}$. The key point is to verify $\bar{x}\in Z_{j+1}$. At the $\omega$-limit point, we find $$\del^c \phi(\bar{x}) \supset \{\bar{y}\} \coprod \del^c \phi(x_0),$$ and the subdifferential at $\bar{x}$ contains at least one new target point $\bar{y} \in Y$. The (UHS) condition implies the average of the projections $\nabla_x \cd(\bar{x}, y_0, \bar{y})$ (with respect to $\bar{\nu}$) is a nonzero vector of $T_{\bar{x}} Z'$, and linearly independent from the nonzero tangent vectors $\nabla_x \cd(\bar{x}, y_1,y_0)$, $y_0, y_1 \in \del^c \phi(x_0)$ (which are orthogonal to $T_{\bar{x}} Z'$). This proves $\bar{x}\in Z_{j+1}$. This proves (a). 

Next we claim the flow defined by $\Psi$ is asymptotically contracting, and therefore $\omega$ is a Lipschitz continuous function. The same arguments from \ref{conc}, \ref{conc1} shows the vector field $\eta_{avg}$ is the gradient of an averaged potential $f_{avg}$, and this potential is asymptotically convex in the direction of $\nabla_x f_{avg}$. This proves (b).

\end{proof}

\begin{proof}[Sketch of proof for Theorem \ref{retract1}] 

%We follow the approach of \cite{Nee}, \S 3. To detail the local construction of $\Psi$, 
We follow the proof of Theorem \ref{retract1}, and construct a continuous vector field $\eta_{avg}(x')$ on $Z' \cap Z_j$ which blows-up precisely on $Z'\cap Z_{j+1} \subset Z' \cap Z_j$, which we assume is nonempty. For initial points $x'$, the maximal interval of existence $I(x')=[0, \omega(x'))$ varies continuously with respect to $x'$ (see Lemma \ref{L;ft}). The field $\eta_{avg}(x')$ will generate a global forward-time continuous mapping $$\Psi:(Z' \cap Z_j) \times I \to Z' \cap Z_j$$ satisfying the usual ordinary differential equation $\frac{d}{ds}[\Psi(x',s')]= \eta_{avg}(\Psi(x',s'))$ for all $s'\in I(x')$. The flow $\Psi$ is reparameterized according to the parameter $\omega(x_0)$ to obtain a continuous deformation retract $\Psi': (Z' \cap Z_j ) \times [0,1] \to Z' \cap Z_{j}$ as desired.
\end{proof}

\begin{proof}
[Proof of Theorem \ref{retract1}]
The (UHS) conditions ensure the cross-differences $\cd$ have nonvanishing gradient $\nabla_x \cd \neq 0$ throughout the domain $Z'$, and the gradients $\nabla_x \cd$ vary continuously over $Z_j$. Moreover the uniform Halfspace condition (UHS2) ensures $||\eta_{avg}(x)||$ is uniformly bounded away from zero in neighborhoods of the poles $\{||\eta_{avg}(x)||=+\infty\}=Z' \cap Z_{j+1}$ in $Z' \cap Z_j$. 

%Indeed, if $\{x_k\}$ is sequence of points in $Z' \cap Z_j$ converging to a point $x_\infty \in Z' \cap Z_{j+1}$, then there exists at least one $y_*$ and open neighborhood $V$ of $y_*$ such that $\eta(x,y)$ diverges to infinity for all $y\in V$. So (UHS2) implies the divergence of $||\eta(x,avg)||$ to $+\infty$ and is uniformly bounded away from zero near the poles. 

%[INCOMPLETE: BELOW]
The vector field $\eta_{avg}(x')$ integrates to a global forward-time flow $\Psi: Z' \cap Z_j \times I \to Z' \cap Z_j$, where as usual we have $d/ds[\Psi(x',s')]=\eta_{avg}(\Psi(x',s'))$ for all $s'\in I(x')$. According to Lemma \ref{L;ft}, for every choice of $x=x(0)$ initial value on $Z'(x)\cap Z_j$, the trajectory $x=x(t)$ converges in finite time to $Z' \cap Z_{j+1}$ with respect to the flow \eqref{ode}. Our Assumptions (A0)--(A6) imply continuous dependence on the choice of initial value. The argument from Proposition \ref{conc1} with Lemma \ref{L;ft} prove the flow $\Psi$ can be reparameterized according to the parameter $\omega(x_0)$ to obtain a continuous deformation retract $\Psi'$ of $Z'\cap Z_j $ onto $Z' \cap Z_j$, as desired.

\end{proof}

Finally we observe the ``local" homotopy isomorphisms $\{\Psi: Z'(x) \cap Z_j \leadsto Z'(x) \cap Z_{j+1}\}$ assemble to a continuous deformation retract of $Z_j \simeq Z_{j+1}$ and establish Theorem \ref{retract2}:

\begin{proof}[Proof of Theorem \ref{retract2}]
The proof of Theorem \ref{retract1} constructs homotopy deformations $Z'(x) \cap Z_j \to Z'(x) \cap Z_{j+1}$, where $Z'(x)=Z(\del \psi^c (x)) \subset Z_1$ for a given $x$ in $Z_j$. When $x$ varies over $Z_j$, the corresponding $Z'=Z'(x)$ are either coincident or intersect along a subset of $Z_{j+1}$. So the local strong deformation retracts $\{\Psi\}$ constructed in Theorem \ref{retract2} assemble to a continuous deformation retract $Z_j \times [0,1] \to Z_j$. Therefore $Z_j \hookleftarrow Z_{j+1}$ is a homotopy isomorphism. Composing the retractions $Z_j \leadsto Z_{j+1}$ for $j=1,\ldots, J$ yields the deformation retract $Z_1 \leadsto Z_{J+1}$, as desired. 
\end{proof}

The above results establish Theorem \ref{B} from the Introduction.

\section{Brief Review of Singularity in the Literature}
The term ``singularity" is evidently overburdened having various interpretations within the literature. We use the term ``singularity" as arising in convex geometry: if $\phi$ is a lower semicontinuous convex function on $\bR^N$, then the ``singularity" of $\phi$ corresponds to the locus of non-differentiability on $dom(\phi)$, i.e. the set of points $x\in dom(\phi)$ where $\del\phi(x)\subset (\bR^N)^*$ is not single-valued. From the viewpoint of optimal transport, 
``singularity" then refers to the locus-of-discontinuity of $c$-optimal semicouplings.

%Our formulation of Kantorovich singularity (Definition \ref{ks}) is both \textit{categorical} and \textit{topological}. %We do not use ``singularity" in the sense of Arnold-Thom catastrophes and caustics [ref]. 

%As elaborated in the Introduction, the locus-of-discontinuity (so defined) is not immediately topological. 

The Kantorovich singularity (Def. \ref{ks}) is both categorical and topological, and provides an alternative view to the so-called ``regularity theory" of optimal transportation. Regularity in optimal transport usually refers to the Monge map $y=T(x)$ defined by the (Twist) condition, and $$T(x'):=\nabla_x c(x',-)^{-1} (-\nabla_x \psi^c(x')).$$ There is large volume of research concerning the $C^{1, \alpha}$ or $C^2$, $C^\infty$ regularity of $T$ under various hypotheses on $c, \sigma, \tau$. C.f. \cite[Ch.12]{Vil1}. This article however is strictly interested in the \emph{discontinuities} of $T$. So we pass silently over questions of the type ``How regular is the map $T$ away from the singularities?". This article and \cite{martel} rather studies the continuity properties of the locus-of-discontinuity of $T$. 

Several results concerning singularities of optimal transports have been attained in the literature. Alberti's paper \cite{Alberti} is basic starting point. Figalli \cite{Fig} studies the singularities of optimal transports between two probability measures supported on bounded open domains in the plane $\bR^2$ with respect to the quadratic Euclidean cost $c(x,y)=||x-y||^2/2$ (equivalently $c=-\langle x,y\rangle$). The main result of \cite[\S 3.2]{Fig} is that the singularity ($Z_2$ in our notation) has topological closure $\overline{Z_2}$ in $\bR^2$ with zero two-dimensional Hausdorff measure, $\mathscr{H}^2(\overline{Z_2})=0.$ Here we remark that the closure $\overline{Z_2}$ is not easily topologized from the singularity $Z_2$. In otherwords the closure is not entirely explicit. Figalli's work was extended in \cite{Fig-Kim}, where the evaluation $\mathscr{H}^n(\overline{Z_2})=0$ was established for singularities of optimal couplings under the hypothesis that probability measures are supported on bounded open domains of $\bR^n$, and again with respect to Euclidean quadratic cost. In \cite{Fig-dP} a similar result is established with respect to more general costs on $\bR^n$ satisfying basic nondegeneracy conditions, namely: 

- the cost $c\in C^2(X\times Y)$ with $||c||_{C^2} <+\infty$, 

- the rule $x\mapsto \nabla_y c(x,y)$ is injective map for every $y$, 

- the rule $y\mapsto \nabla_x c(x,y)$ is injective map for every $x$, and 

- $det(D_{xy}^2 c) \neq 0$ for all $(x,y)$. 

The investigations of Figalli suggest the following: if the source and target measures $\sigma, \tau$ are absolutely continuous to $\mathscr{H}_X, \mathscr{H}_Y$, and the cost $c$ satisfies (A0)--(A4), then the singularity $Z_2$ has Hausdorff dimension satisfying $$\dim_{\mathscr{H}} (Z_2)\leq \dim(Y).$$ We motivate this dimension formula in \ref{sing-dim}. The recent work \cite{KitMc} confirms this estimate under some particular conditions, namely Euclidean quadratic cost, convex source, and convexity and affine-independence of the disjoint target components. %The goal of the present thesis is to establish topological properties of the singularities of optimal transports. This problem has not appeared in the literature, and we believe it is important area of study, and especially for applications to various topological problems (e.g. explicit souls and spines). 

 %These conditions are generalized in 

% We developed the method of \cite{KitMc} in our Proposition \ref{sing-dim}. The Proposition \ref{sing-dim} confirms this popular expectation of the dimensions of singularities of optimal transports.

%Basic ideas on discontinuity of optimal transports can be found in \cite[Ch.12]{Vil1}. In the same chapter the reader finds the following quote following Loeper's fundamental counterexample (\cite[Thm.12.4]{Vil1}): ``...there is no hope for general regularity results outside the world of nonnegative sectional curvature". Our thesis happily overturns this cynicism. Indeed optimal transports in nonnegative curvature have persistent inescapable singularities. But the singularities themselves are topologically well-defined, and not pathological subsets of optimal transports, as guaranteed by the relatively mild Assumptions (A0)--(A4) and (A5) we have described. Once one embraces the persistence of discontinuity in nonpositive curvature (e.g. the large support of the contravariant functor $Z: 2^Y \to 2^X$), then one discovers their highly regular topological nature. 

%There is an important ``dimension estimate" concerning singularities within the literature. 

%Under general assumptions on the cost, one expects the singularity $Z_2$ has Hausdorff dimension $\dim_{\mathscr{H}} Z_2\leq \dim(X)-\dim(Y)-1$. Under Assumptions (A0)--(A3), we confirm this Hausdorff dimension in Proposition \ref{sing-dim}.

\printbibliography[title={References}]
\end{document}